\documentclass[a4paper,10pt]{article}
\usepackage{amsmath}
\usepackage{amssymb}
\usepackage{amsthm}
\usepackage{amsfonts}
\usepackage{array}
\usepackage{enumerate}
\usepackage{xcolor}
\usepackage{bbm}
\usepackage{esint}
\usepackage{graphicx}
\usepackage{float}
\usepackage{setspace}
\usepackage{authblk}
\usepackage[margin=3cm]{geometry}

\newtheorem{theorem}{Theorem}
\newtheorem{definition}[theorem]{Definition}
\newtheorem*{definition*}{Definition}
\newtheorem*{theorem*}{Theorem}
\newtheorem{lemma}[theorem]{Lemma}
\newtheorem*{lemma*}{Lemma}
\newtheorem{corollary}[theorem]{Corollary}
\newtheorem{example}[theorem]{Example}

\newtheorem{remark}[theorem]{Remark}
\newtheorem*{remark*}{Remark}

\newtheorem{proposition}[theorem]{Proposition}

\newcommand{\wsc}{\overset{*}{\rightharpoonup}}
\newcommand{\norm}[1]{\left\|{#1}\right\|}
\newcommand{\ip}[2]{\left<#1,#2\right>}
\newcommand\restrict[2]{\ensuremath{\left.#1\right|_{#2}}}
\newcommand{\dd}{\mathrm{d}}
\newcommand{\mbR}{\mathbb{R}}
\newcommand{\mbfM}{\mathbf{M}}

\newcommand{\BV}{\mathrm{BV}}
\newcommand{\Lp}{\mathrm{L}}
\newcommand{\W}{\mathrm{W}}
\newcommand{\C}{\mathrm{C}}
\newcommand{\pd}{\partial}
\newcommand{\sbullet}{\begin{picture}(1,1)(-0.5,-2)\circle*{2}\end{picture}}
\newcommand{\frarg}{\,\sbullet\,}

\def\XXint#1#2#3{{\setbox0=\hbox{$#1{#2#3}{\int}$ }
\vcenter{\hbox{$#2#3$ }}\kern-.6\wd0}}

\DeclareMathOperator{\supp}{supp}

\renewcommand{\restrict}{\begin{picture}(10,8)\put(2,0){\line(0,1){7}}\put(1.8,
0){\line(1,0){7}}\end{picture}}
\author{Filip Rindler\footnote{Mathematics Institute, University of Warwick, Coventry CV4 7AL, United Kingdom, and University of Cambridge (on leave), Gonville \& Caius College, Trinity Street, Cambridge CB2 1TA, United Kingdom; \texttt{F.Rindler@warwick.ac.uk}.} \and Giles Shaw\footnote{Cambridge Centre for Analysis, University of Cambridge, Centre for Mathematical Sciences, Wilberforce Road, Cambridge CB3 0WA, United Kingdom; \texttt{G.W.H.Shaw@maths.cam.ac.uk}.}}

\title{Strictly continuous extension of functionals with linear growth to the space $\BV$}
\begin{document}
 \maketitle
\begin{abstract}
In this paper, we prove that the integral functional $\mathcal{F}[u]\colon \BV(\Omega;\mbR^m)\to\mbR$ defined by
\[
\mathcal{F}[u]:=\int_\Omega f(x,u(x),\nabla u(x))\;\dd x+ \int_\Omega\int_0^1 f^\infty\left(x,u^\theta (x),\frac{\dd D^s u}{\dd|D^s u|}(x)\right)\;\dd\theta\;\dd|D^su|(x)
\]
is continuous over $\BV(\Omega;\mbR^m)$, with respect to the topology of area-strict convergence, a topology in which $(\W^{1,1}\cap \C^\infty)(\Omega;\mbR^m)$ is dense. This provides conclusive justification for the treatment of $\mathcal{F}$ as the natural extension of the functional
\[
 u\mapsto\int_\Omega f(x,u(x),\nabla u(x))\;\dd x,
\]
defined for $u\in \W^{1,1}(\Omega;\mbR^m)$. This result is valid for a large class of integrands satisfying $|f(x,y,A)|\leq C(1+|y|^{d/(d-1)}+|A|)$ and its proof makes use of Reshetnyak's Continuity Theorem combined with a lifting map $\mu[u]\colon \BV(\Omega;\mbR^m)\to\mbfM(\Omega\times\mbR^m;\mbR^{m\times d})$. To obtain the theorem in the case where $f$ exhibits $d/(d-1)$ growth in the $y$ variable, an embedding result from the theory of concentration-compactness is also employed.
\end{abstract}

\section{Introduction}
In the Calculus of Variations, the chief obstruction to the application of the Direct Method in studying the existence of solutions to the problem
\[
\min_{u\in \W^{1,1}(\Omega;\mbR^m)}\mathcal{F}[u]:=\min_{u\in \W^{1,1}(\Omega;\mbR^m)}\int_\Omega f\left(x,u(x),\nabla u(x)\right)\;\dd x
\]
is the fact that norm-bounded subsets of the space $\W^{1,1}(\Omega;\mbR^m)$ fail to exhibit any kind of good compactness property. This is in direct contrast to the situation where $\mathcal{F}$ is to be minimised over $\W^{1,p}(\Omega;\mbR^m)$ for $p>1$, in which case reflexivity of $\W^{1,p}(\Omega;\mbR^m)$ (or the Banach--Alaoglu Theorem when $p=\infty$) ensures that norm-bounded sets are weakly (weakly* for $p=\infty$) compact. If $f$ is assumed to be coercive, minimising sequences $(u_j)\subset \W^{1,p}(\Omega;\mbR^m)$ of $\mathcal{F}$ must be norm bounded and hence, by weak relative compactness of bounded sets in $\W^{1,p}(\Omega;\mbR^m)$, can be assumed to converge to a limit $\overline{u}\in \W^{1,p}(\Omega;\mbR^m)$, which is then a natural candidate for a global minimiser of $\mathcal{F}$. Since this argument is not applicable in $\W^{1,1}(\Omega;\mbR^m)$, the domain of $\mathcal{F}$ must be extended from $\W^{1,1}(\Omega;\mbR^m)$ to a larger function space with better compactness properties. In many cases, the right choice here is the space $\BV(\Omega;\mbR^m)$ of functions of bounded variation, a function space which admits a weak* topology under which $(\W^{1,1}\cap \C^\infty)(\Omega;\mbR^m)$ is a sequentially dense subspace of $\BV(\Omega;\mbR^m)$ and sequential weak* compactness holds.

Having obtained the existence of candidate minimisers in $\BV(\Omega;\mbR^m)$, the next step in the application of the Direct Method is to examine when $\mathcal{F}$ satisfies the lower semicontinuity property $\mathcal{F}[u]\leq\liminf_{j}\mathcal{F}[u_j]$ for every sequence $(u_j)\subset \BV(\Omega;\mbR^m)$ such that $u_j\wsc u$. In order to do this, a suitable extension of $\mathcal{F}$ to $\BV(\Omega;\mbR^m)$ must be identified so that a value can be assigned to $\mathcal{F}[u]$ for $u\in (\BV\setminus \W^{1,1})(\Omega;\mbR^m)$. There is no unique extension of $\mathcal{F}$ from $\W^{1,1}(\Omega;\mbR^m)$ to $\BV(\Omega;\mbR^m)$, and so a criterion is needed to identify the `right' extension in this context for as wide a class of integrands $f$ as possible. In general, we cannot hope to obtain $\mathcal{F}$ as the weakly* continuous extension to $\BV(\Omega;\mbR^m)$ of $u\mapsto\int_\Omega f(x,u(x),\nabla u(x))\;\dd x$: Example~\ref{areastrictbeatsstrict} below demonstrates a weakly* convergent sequence in $\W^{1,1}(\Omega;\mbR^m)$ under which the map $u\mapsto \int_\Omega\sqrt{1+|u'(x)|^2}\;\dd x$ fails to converge. A priori, it is far from clear how one might extend $\mathcal{F}$ in such a way that every $u\in \BV(\Omega;\mbR^m)$ has at least one recovery sequence $(u_j)\subset \W^{1,1}(\Omega;\mbR^m)$ (i.e. a sequence $(u_j)$ such that $u_j\wsc u$ and $\mathcal{F}[u_j]\to\mathcal{F}[u]$): the derivative $Du$ of a function $u\in \BV(\Omega;\mbR^m)$ is defined only as a (potentially Lebesgue-singular) matrix-valued measure, in which case the expression $\int_\Omega f(x,u(x),Du)\;\dd x$ is not well-defined.

The \emph{relaxation method} (essentially first proposed and implemented by Serrin, see~\cite{SERR59NDI} and~\cite{SERR61PVI}) is often used to extend $\mathcal{F}$ for a restricted class of integrands $f$: If $u$ is scalar-valued and $f(x,y,\frarg)$ is convex, or if $u$ is vector-valued and $f=f(x,A)$ and $f(x,\frarg)$ is quasiconvex it can be shown (see, for instance,~\cite{AmCiFu07RRBV,AmbDal92RBVQ,DMas79IRBV,KriRin10RSI,Rind12LSYM}) that the weak* relaxation of $\mathcal{F}$ to $\BV(\Omega;\mbR^m)$,
\begin{equation}\label{Fasrel}
\mathcal{F}_{**}[u]:=\inf\left\{\liminf_{j\to\infty}\mathcal{F}[u_j]:u_j\wsc u,\,(u_j)\subset \W^{1,1}(\Omega;\mbR^m)\right\},
\end{equation}\
 admits the integral representation
\begin{equation}\label{functionalequation}
\mathcal{F}[u]=\mathcal{F}_{**}[u]=\int_\Omega f\left(x,u(x),\nabla u(x)\right)\;\dd x +\int_\Omega\int_0^1 f^\infty\left(x,u^\theta(x),\frac{\dd D^su}{\dd|D^su|}(x)\right)\;\dd\theta\;\dd|D^su|(x)
\end{equation}
for $u\in \BV(\Omega;\mbR^m)$ (definitions for $u^\theta$ and $f^\infty$ can be found in Section~\ref{prelims}). In the general case where $u$ is vector-valued, $f=f(x,y,A)$ and $f(x,y,\frarg)$ is quasiconvex (see~\cite{FonMul93RQFB}), the representation~\eqref{functionalequation} fails and must be replaced by a more general expression where the density for the $\mathcal{H}^{d-1}$-absolutely continuous part of $\mathcal{F}$ can only be identified as the solution to a specific cell problem which does not always coincide (even when $f(x,y,\frarg)$ is convex, see~\cite{AmbPal93IRRF}) with the $\mathcal{H}^{d-1}$-density of~\eqref{functionalequation}. 

As defined in \eqref{functionalequation}, $\mathcal{F}[u]$ is equal to our original $\mathcal{F}[u]$ for $u\in \W^{1,1}(\Omega;\mbR^m)$ and is therefore an extension of the original $\mathcal{F}$ to $\BV(\Omega;\mbR^m)$. It follows from \eqref{Fasrel} that, at least in the scalar or $u$-independent case, each $u\in \BV(\Omega;\mbR^m)$ admits a recovery sequence $(u_j)\subset(\W^{1,1}\cap \C^\infty)(\Omega;\mbR^m)$, which implies that \eqref{functionalequation} meets the minimum criteria for a suitable extension of $\mathcal{F}$. In general, however, one is unable to say anything about the properties of such a recovery sequence or if a better extension of $\mathcal{F}$ exists which admits strictly more recovery sequences.

Since the restriction of $\mathcal{F}_{**}$ to $\W^{1,1}(\Omega;\mbR^m)$ is lower semicontinuous with respect to weak convergence in $\W^{1,1}(\Omega;\mbR^m)$, it can only be used to extend $\mathcal{F}$ in situations where $\mathcal{F}$ is also lower semicontinuous over $\W^{1,1}(\Omega;\mbR^m)$ (i.e., when $f(x,y,\frarg)$ is convex/quasiconvex) and so the relaxation method cannot be used to extend $\mathcal{F}$ for general integrands. As defined in \eqref{functionalequation}, however, the restriction of $\mathcal{F}[u]$ to $\W^{1,1}(\Omega;\mbR^m)$ is always equal to $\int_\Omega f(x,u(x),\nabla u(x))\;\dd x$, regardless of the convexity properties of $f$. This suggests that, if the extension given by \eqref{functionalequation} still admits $\W^{1,1}(\Omega;\mbR^m)$-recovery sequences, it can be taken as a candidate functional for the extension of $\mathcal{F}$ to $\BV(\Omega;\mbR^m)$ even when $f$ is not convex. In order to justify this position, we must find a way of showing that the extension given by \eqref{functionalequation} always admits $\W^{1,1}(\Omega;\mbR^m)$-recovery sequences and argue that no better extension is to be found.

This paper is primarily devoted to proving the following theorem, which establishes that \eqref{functionalequation} defines an extension of $\mathcal{F}$ valid for \emph{general} integrands $f$ in a way that satisfies all of the requirements above:
\begin{theorem}\label{unbddcontinuitythm}
 Let $\Omega\subset\mbR^d$ be a bounded domain with compact Lipschitz boundary, define $1^*:=d/(d-1)$ ($1^*=\infty$ if $d=1$) and let $p\in[1,1^*]$ if $d\geq 2$ and $p\in[1,1^*)$ if $d=1$. Let $f\in \C(\Omega\times\mbR^m\times\mbR^{m\times d})$ satisfy the requirements
\begin{equation}\label{assumption1}
|f(x,y,A)|\leq C(1+|y|^p+|A|) \quad\text{ for all } (x,y,A)\in\Omega\times\mbR^m\times\mbR^{m\times d}
\end{equation}
and
\begin{equation}\label{assumption2}
 f^\infty\text{ exists}.
\end{equation}
Then, the functional
\begin{equation*}
 \mathcal{F}[u]:=\int_\Omega f\left(x,u(x),\nabla u(x)\right)\;\dd x +\int_\Omega\int_0^1 f^\infty\left(x,u^\theta(x),\frac{\dd D^su}{\dd|D^su|}(x)\right)\;\dd\theta\;\dd|D^su|(x)
\end{equation*}
is area-strictly continuous on $\BV\left(\Omega;\mbR^m\right)$.
\end{theorem}
Here, area-strict convergence (defined in Section~\ref{prelims}) is a notion of convergence with respect to which $(\W^{1,1}\cap \C^\infty)(\Omega;\mbR^m)$ is dense in $\BV(\Omega;\mbR^m)$ and which implies weak* convergence. Every area-strictly convergent sequence is thus a recovery sequence and, by area-strict density of $(\W^{1,1}\cap \C^\infty)(\Omega;\mbR^m)$ in $\BV(\Omega;\mbR^m)$, \eqref{functionalequation} is the \emph{unique} extension of $\mathcal{F}$ to $\BV(\Omega;\mbR^m)$ for which this holds. Related results can be found in~\cite{Dell91LSCF} and Theorem~3 in~\cite{KriRin10RSI}.

A surprising implication of Theorem~\ref{unbddcontinuitythm} and the failure of the representation~\eqref{functionalequation} for the case $f=f(x,y,A)$, $m>1$, is that, in contrast to the situation where $f=f(x,A)$ or $m=1$, the relaxation $\mathcal{F}_{**}$, cannot be area-strictly continuous in general, not even when $f(x,y,\frarg)$ is convex. Conversely, it must also be the case that, even when $f(x,y,\frarg)$ is convex or quasiconvex, $\mathcal{F}$ is not always weakly* lower semicontinuous over $\BV(\Omega;\mbR^m)$, despite being area-strictly continuous over $\BV(\Omega;\mbR^m)$ and weakly* lower semicontinuous over $\W^{1,1}(\Omega;\mbR^m)$.

The Rellich-Kondrachov Theorem for $\BV(\Omega;\mbR^m)$ states that $\BV(\Omega;\mbR^m)\hookrightarrow \Lp^p(\Omega;\mbR^m)$ for $p\in[1,1^*]$ and it is known that this embedding result is sharp, so that $\BV(\Omega;\mbR^m)$ cannot be embedded into any higher $\Lp^p$ space. Hence, the growth hypothesis $|f(x,y,A)|\leq C(1+|y|^{1^*}+|A|)$ in Theorem~\ref{unbddcontinuitythm} is optimal, in that it represents the weakest natural condition necessary to ensure that $\mathcal{F}$ is finite on $\BV(\Omega;\mbR^m)$. It might seem natural that the result holds for $f$ satisfying $|f(x,y,A)|\leq C(1+|y|^p+|A|)$ when $p< 1^*$, as a consequence of the fact that in this case the embedding $\BV(\Omega;\mbR^m)\hookrightarrow \Lp^p(\Omega;\mbR^m)$ is compact. That (for $d>1$) this result is true even when $p=1^*$ is surprising and the proof in this case makes use of Lions' concentration compactness principle. For the case $d=1$ where $1^*=\infty$, Example \ref{excriticalcasefails} demonstrates that the theorem does not hold. We will also show that this result holds true for Carath\'eodory $f$, so long as the recession function $f^\infty$ remains continuous on $(\Omega\setminus N)\times\mbR^m\times\mbR^{m\times d}$, where $N$ is $\mathcal{H}^{d-1}$ negligible, see Theorem~\ref{notcts}.

The organisation of this paper is as follows: First, necessary preliminaries are introduced in Section~\ref{prelims}, including definitions for the recession function $f^\infty$, the jump averaging function $u^\theta$ and the metrics of strict and area-strict convergence on $\BV(\Omega;\mbR^m)$. Then, the main tools for the proof of Theorem~\ref{unbddcontinuitythm} are introduced in Section~\ref{secliftings}: the lifting $\mu[u]\in \mbfM(\Omega\times\mbR^m;\mbR^{m\times d})$ of a function $u\in \BV(\Omega;\mbR^m)$ is defined, the strict continuity of the map $u\mapsto\mu[u]$ is proved and a proposition is established stating that the embedding
\[
 \left(\BV(\Omega;\mbR^m),\;\textbf{strict}\right)\hookrightarrow \Lp^{d/(d-1)}(\Omega;\mbR^m)
\]
is continuous (here, \textbf{strict} denotes the topology induced by strict convergence in $\BV(\Omega;\mbR^m)$, see Section~\ref{prelims}). The perspective integrand $\tilde{f}$ of $f$ is then introduced in Section~\ref{secperints} which, together with the properties of liftings obtained in Section~\ref{secliftings} and Reshetnyak's Continuity Theorem, is used to prove Theorem~\ref{unbddcontinuitythm} in the case where $f$ is bounded in the middle variable. The full version of Theorem~\ref{unbddcontinuitythm} is then established in Section~\ref{bigproof} via an approximation argument combined with the strict continuity of the embedding $\BV(\Omega;\mbR^m)\hookrightarrow \Lp^{d/(d-1)}(\Omega;\mbR^m)$. Finally, the continuity assumptions on $f$ and $f^\infty$ are weakened and an example is provided to show that this refinement of the theorem is optimal.

\subsection{Funding}
This work was supported by the UK Engineering and Physical Sciences Research Council (EPSRC) [EP/H023348/1 for the University of Cambridge Centre for Doctoral Training, the Cambridge Centre for Analysis, to G.S., EP/L018934/1 to F.R.]; and the University of Warwick.

\subsection{Acknowledgements}
The authors would like to thank Jan Kristensen for many insightful conversations related to this paper as well as Helge Dietert, Tom Holding and Marcus Webb for comments and remarks.

\section{Preliminaries}\label{prelims}
\subsection{Facts about $\Lp^p$ and the recession function}
Throughout this paper, $\Omega\subset\mbR^d$ will be assumed to be a connected open set with compact Lipschitz boundary (in fact, $\Omega$ need only be a connected bounded extension domain, but we omit this throughout the paper for ease of reading) and $\mathbb{B}^{m\times d}$ will denote the open unit ball in $\mbR^{m\times d}$	. The constant $1^*$ is defined to  be the \textbf{critical Sobolev embedding exponent}
\[
 1^*:=\begin{cases}
      \frac{d}{d-1}&\text{ if }d>1,\\
      \infty &\text{ if }d=1.
      \end{cases}
\]

We recall here some technical facts about weak and norm convergence in $\Lp^p$ spaces which will be used in the sequel:

\begin{lemma}[Brezis--Lieb Lemma]
 Let $1\leq p<\infty$ and let $(u_j),u$ be functions in $\Lp^p(\Omega;\mbR^m)$ such that
\[
u_j\rightharpoonup u\qquad\text{ and }\qquad u_j\to u \text{ pointwise a.e.}
\]
Then we have that
\[
 \lim_{j\to\infty}\norm{u_j}_p^p=\norm{u}_p^p+\lim_{j\to\infty}\norm{u_j-u}_p^p.
\]
\end{lemma}
A proof of this result can be found in~\cite{BreLie83RBC} and also in~\cite{Evan90WCMN}.

\begin{theorem}[Scorza Dragoni]
 Let $f\colon\Omega\times\mbR^m\times\mbR^{m\times d}\to\mbR$ be Carath\'eodory. Then for every $\varepsilon>0$, there exists a compact set $K_\varepsilon\subset\Omega$ such that $\mathcal{L}^d(\Omega\setminus K_\varepsilon)<\varepsilon$ and $f\restrict(K_\varepsilon\times\mbR^m\times\mbR^{m\times d})$ is continuous.
\end{theorem}
For a proof, see~\cite[p.~74]{Daco08DMCV}.

We now define the recession function $f^\infty$ of an integrand $f$, whose purpose is to capture information about the behaviour of $f(x,y,A)$ as $|A|\to\infty$. Note that we require our definition of the recession function to be more restrictive than what is usually found in the literature (where only the existence of $\lim_{t\to\infty}f(x,y,tA)/t$ or $\limsup_{t\to\infty}f(x,y,tA)/t$ is assumed).
\begin{definition}[The recession function]
 For $f\in \C\left(\Omega\times\mbR^m\times\mbR^{m\times d}\right)$, define the recession function $f^\infty\in \C(\Omega\times\mbR^m\times\mbR^{m\times d})$ of $f$ by
\begin{equation*}
 f^\infty\left(x,y,A\right)=\lim_{\substack{x'\to x\\y'\to y\\ A'\to A\\t\to\infty}}\frac{f\left(x',y',tA'\right)}{t},
\end{equation*}
whenever the right hand side exists independently of the order in which the limits of the individual sequences are taken and of which sequences are used.
\end{definition}
The definition of the recession function implies that, whenever it exists, it must be continuous. This property is necessary in order for the function $\tilde{f}$ defined in the proof of Lemma~\ref{continuitythm} to be continuous which, in turn is necessary for Reshetnyak's Continuity Theorem to hold. For further intricacies related to the definition of the recession function, we refer to~\cite{Rind12LSYM}. Note that the recession function is \textbf{positively} $1$-\textbf{homogeneous} in the final variable, that is $f^\infty(x,y,\lambda A)=\lambda f^\infty(x,y,A)$ for each $\lambda\geq 0$.

\subsection{Facts about measures}
We will denote by $\mathcal{B}(X)$ the Borel $\sigma$-algebra on a normed space $X$ and the space of $\mbR^{m\times d}$-valued Radon measures acting on $X$ (we will always take $X=\Omega$, $X=\mbR^m$ or $X=\Omega\times\mbR^m$) by $\mbfM(X;\mbR^{m\times d})$. The spaces of scalar-valued and positive measures on $X$ will be denoted by $\mbfM(X)$ and $\mbfM^+(X)$ respectively. A sequence of measures  $(\mu_j)$ is said to converge \textbf{strictly} to $\mu$ if $\mu_j\wsc\mu$ and $|\mu_j|(X)\to|\mu|(X)$, where $|\mu|$ is the total variation measure of $\mu$. We will denote the Radon-Nikodym derivative (or polar function) of a measure $\mu$ with respect to its total variation by $\frac{\dd\mu}{\dd|\mu|}$.

The following theorems concerning the convergence of nonlinear functionals of measures will be of great importance:

\begin{theorem}[Reshetnyak's Lower Semicontinuity Theorem]\label{reshetnyaklowersemicontinuity}
Let $\mu,(\mu_j)_{j\in\mathbb{N}}$ be $\mbR^{m\times d}$-valued finite Radon measures on  $\Omega\times\mbR^m$. If $\mu_j\wsc\mu$ then
\[
 \int_{\Omega\times\mbR^m} f\left(x,y,\frac{\dd\mu}{\dd|\mu|}(x,y)\right)\;\dd|\mu|(x,y)\leq\liminf_{j\to\infty}\int_{\Omega\times\mbR^m} f\left(x,y,\frac{\dd\mu_j}{\dd|\mu_j|}(x,y)\right)\;\dd|\mu_j|(x,y)
\]
for every lower semicontinuous function $f\colon\Omega\times\mbR^m\times\mbR^{m\times d}\to[0,\infty]$ which is positively 1-homogenous and convex in the last variable.
\end{theorem}

\begin{theorem}[Reshetnyak's Continuity Theorem]\label{Reshetnyak}
Let $f\in
\C\left(\Omega\times\mbR^m\times\pd\mathbb{B}^{m\times d}\right)$ be bounded. Then, the map
\begin{equation*}
 \mu \mapsto \int_{\Omega\times\mbR^m} f\left(x,y,\frac{\dd\mu}{\dd|\mu|}(x,y)\right)\dd|\mu|(x,y)
\end{equation*}
is continuous on $\mbfM(\Omega\times\mbR^m;\mbR^{m\times d})$ with respect to strict
convergence.
\end{theorem}

Given a measure $\mu\in\mbfM^+(\Omega)$, we say that $\nu$ is a \textbf{$\mu$-measurable $\mbfM(\mbR^m;\mbR^{m\times d})$-valued map} or \textbf{parametrised measure} if $\nu\colon x\mapsto\nu_x$ is a function $\nu\colon\Omega\to\mbfM(\mbR^m;\mbR^{m\times d})$ such that the map $x\mapsto\nu_x(A)$ is $\mu$-measurable for every fixed $A\in\mathcal{B}(\mbR^m)$.

If $\mu\in\mbfM^+(\Omega)$ and $\nu\colon\Omega\to\mbfM(\mbR^m;\mbR^{m\times d})$ is a $\mu$-measurable parametrised measure, we can define the \textbf{generalised product} $\mu\otimes\nu\in\mbfM(\Omega\times\mbR^m;\mbR^{m\times d})$ of $\mu$ and $\nu$ by its action on elements of $\C_0(\Omega\times\mbR^m)$:
\[
\int_{\Omega\times\mbR^m}\varphi(x,y)\;\dd(\mu\otimes\nu)(x,y):=\int_\Omega\left(\int_{\mbR^m}\varphi(x,y)\;\dd\nu_x(y)\right)\;\dd\mu(x).
\]
The following theorem lets us decompose a measure defined on $\Omega\times\mbR^m$ into a generalised product involving the projection of its total variation.
\begin{theorem}[Disintegration of measures]
 Let $\eta\in\mbfM(\Omega\times\mbR^m;\mbR^{m\times d})$ and let $\pi\colon\Omega\times \mbR^m\to \Omega$ be the projection
operator defined by $\pi(x,y)=x$ for $(x,y)\in\Omega\times\mbR^m$. Then there exists a $\pi_\sharp|\eta|$-almost everywhere unique
measure-valued map $\nu\colon\Omega\to \mbfM(\mbR^m;\mbR^{m\times d})$ such that each
$|\nu_x|$ is a probability measure and
$\eta=\left(\pi_\sharp|\eta|\right)\otimes\nu$, where $\pi_\sharp|\eta|\in\mbfM(\Omega\times\mbR^m;\mbR^{m\times d})$ is uniquely defined by $\pi_\sharp|\eta|(A\times B):=|\eta|(A)$. Furthermore,
$|\eta|=\left(\pi_\sharp|\eta|\right)\otimes|\nu|$ (where $|\nu|$ is defined by
$|\nu|_x=|\nu_x|$) and, up to scaling, this is the only way of factoring $\nu$ over $\Omega$ and $\mbR^m$: if $\eta=\xi\otimes\nu$ where $\xi\in\mbfM^+(\Omega)$ and $|\nu|\colon F\to\mbfM^1(\mbR^m)$, then $\pi_\sharp|\eta|=\xi$.
\end{theorem}

The standard reference for all of the above is~\cite{AmFuPa00FBVF}, although we note that new proofs of Reshetnyak's theorems can be found in~\cite{Spec11SPRR}.
\subsection{Facts about $\BV(\Omega;\mbR^m)$}
Recall that the function space $\BV(\Omega;\mbR^m)$ is defined as the space of $\Lp^1(\Omega;\mbR^m)$ functions whose distributional derivatives are measures in $\mbfM(\Omega;\mbR^{m\times d})$. For a given $\BV(\Omega;\mbR^m)$ function $u$, the domain $\Omega$ admits the following decomposition into disjoint sets:
\[
 \Omega=\mathcal{D}_u\cup\mathcal{J}_u\cup\mathcal{C}_u\cup\mathcal{N}_u,
\]
where $\mathcal{D}_u$ denotes the set of points at which $u$ is approximately differentiable, $\mathcal{J}_u$ denotes the set of jump points of $u$, $\mathcal{C}_u$ denotes the set of points where $u$ is approximately continuous but not approximately differentiable and $\mathcal{N}_u$ is a set satisfying $\mathcal{H}^{d-1}(\mathcal{N}_u)=0$, where $\mathcal{H}^{d-1}$ is the $(d-1)$-dimensional \textbf{Hausdorff measure}. We have that $\mathcal{L}^d(\Omega\setminus\mathcal{D}_u)=0$ and that $\mathcal{J}_u$ is a $\mathcal{H}^{d-1}$-rectifiable set. The derivative $Du$ of $u$ can then be written as a sum of mutually singular measures,
\[
 Du=\nabla u\mathcal{L}^{d}+D^j u+ D^c u=\nabla u\mathcal{L}^d+D^s u,
\]
where $\nabla u$ is the approximate derivative of $u$, $D^ju=Du\restrict\mathcal{J}_u$, $D^c u=Du\restrict\mathcal{C}_u$ and $D^s u=D^ju+D^cu$. An important consequence of this decomposition is the fact that
\[
 \mathcal{H}^{d-1}(B)=0\implies|Du|(B)=0.
\]
Recall also that, viewed as an element of $\Lp^1(\Omega;\mbR^m)$, $u$ admits a representative, $\tilde{u}$, known as the \textbf{precise representative} which is approximately continuous $\mathcal{H}^{d-1}$-almost everywhere in $\Omega$.

\begin{definition}[The jump averaging function]\label{jumpaverage}
For a given function $u\in \BV\left(\Omega;\mbR^m\right)$, define its
\textbf{jump averaging function}, $u^\theta\colon \Omega\times [0,1]\to\mbR^m$, for $\mathcal{H}^{d-1}$-almost every $x\in\Omega$, by
\begin{equation*}
u^\theta(x):=
\begin{cases}
\theta u^+(x)+(1-\theta)u^-(x) & \text{ for }x\in\mathcal{J}_u, \\
\tilde{u}(x) & \text{ for }x\in\mathcal{D}_u\cup\mathcal{C}_u,
\end{cases}
\end{equation*}
where $u^+$ and $u^-$ are the upper and lower limits of $u$ on $\mathcal{J}_u$ and $\tilde{u}$ is the precise representative.
\end{definition}
Strictly, $u^\theta$ is ill-defined, since $u^+$, $u^-$ are only defined up to
permutation. However, we will only make use of $u^\theta$ in expressions of the
form
\[
 \int_0^1 g\left(u^\theta(x)\right)\;\dd\theta
\]
which are invariant under our choice of $u^+,u^-$, so no issues will arise from this
ambiguity.

\begin{definition}\label{volpert}
Given a continuously differentiable Lipschitz function $f\colon\mbR^m\to\mbR^n$ and a function $u\in
\BV\left(\Omega;\mbR^m\right)$, we define the \textbf{Vol'pert averaged
superposition}, $\overline{f}_u$, of $f$ by
\[
 \overline{f}_u(x):=\int_0^1 \nabla f\left(u^\theta(x)\right)\;\dd\theta.
\]
\end{definition}

\begin{theorem}[The chain rule in $\BV(\Omega;\mbR^m)$]\label{chainrule}
 Let $f\colon\mbR^m\to\mbR^n$ be continuously differentiable and Lipschitz and let $u\in
\BV\left(\Omega;\mbR^m\right)$. Then $v:=f\circ u\in
\BV\left(\Omega;\mbR^n\right)$ and $Dv$ is given by
\begin{align*}
\nabla v\,\mathcal{L}^d\restrict\Omega & =(\nabla f)(\tilde{u})\,\nabla
u\,\mathcal{L}^d\restrict\Omega, \\
D^j v &=\left(f(u^+)-f(u^-)\right)\otimes\tau_u\,\mathcal{H}^{d-1}\restrict\mathcal{J}_u,
\\
D^c v & =\nabla f\left(\tilde{u}\right)D^c u,
\end{align*}
where $\tau_u$ is the \textbf{jump direction} of $u$: that is, the orientation vector of $\mathcal{J}_u$ (see, for example,~\cite{AmFuPa00FBVF}, Theorem~3.78).
\end{theorem}
Since, for $x\in\mathcal{J}_u$,
\begin{equation}\label{eqchainruleidentity}
\frac{\dd}{\dd\theta}f\left(u^\theta(x)\right)=\nabla
f\left(u^\theta(x)\right)\left(u^+(x)-u^-(x)\right),
\end{equation}
we can summarise
the chain rule in $\BV(\Omega;\mbR^m)$ concisely as the statement that
\begin{equation*}
 D\left(f\circ u\right)=\overline{f}_uDu.
\end{equation*}

\begin{theorem}[Rellich--Kondrachov]
When both spaces are endowed with their norm topology, the embedding $\BV(\Omega;\mbR^m)\hookrightarrow \Lp^p(\Omega;\mbR^m)$ for $p\in[1,1^*]$ is continuous. For $p\in[1,1^*)$, the embedding is compact.
\end{theorem}

\subsection{Strict and area-strict convergence}
We will now introduce two metrics on $\BV(\Omega;\mbR^m)$, the strict metric and the area-strict metric. Our interest in these two metrics stems from the fact that they induce a topology which is stronger than the weak* topology on $\BV(\Omega;\mbR^m)$, yet with respect to which $\W^{1,1}(\Omega;\mbR^m)$ and $\C^\infty(\Omega;\mbR^m)$ functions are dense.
\begin{definition}[Strict convergence]
 We say that a sequence $(u_j)\subset\BV(\Omega;\mbR^m)$ converges \textbf{strictly} to $u\in\BV\left(\Omega;\mbR^m\right)$ if it does so with respect to the metric
\begin{equation*}
 d(u,v):=\norm{u-v}_{\Lp^1}+||Du|\left(\Omega\right)-|Dv|\left(\Omega\right)|,
\end{equation*}
so that $(u_j)$ converges strictly to $u$ in $\BV$ if and only if $u_j\to u$ in
$\Lp^1$ and $|Du_j|\left(\Omega\right)\to |Du|\left(\Omega\right)$.
\end{definition}
Strictly convergent sequences are norm-bounded in $\BV(\Omega;\mbR^m)$, which implies that
they have weakly* convergent subsequences. Using the sequential weak* compactness of bounded sets in $\BV(\Omega;\mbR^m)$, we deduce that strict convergence of a sequence $(u_j)$ in $\BV\left(\Omega;\mbR^m\right)$ implies strict convergence of $(Du_j)$ in $\mbfM\left(\Omega;\mbR^m\right)$, i.e., $Du_j\wsc Du$ and $|Du_j|(\Omega)\to|Du|(\Omega)$. 
 
\begin{definition}[Area-strict convergence]
 Define \textbf{area-strict convergence} on $\BV(\Omega;\mbR^m)$ via the metric
\[
 d(u,v):=\norm{u-v}_{\Lp^1}+\left|\int_\Omega\sqrt{1+|\nabla u|^2}\;\dd x+|D^su|\left(\Omega\right)-\left(\int_\Omega\sqrt{1+|\nabla
v|^2}\;\dd x +|D^sv|\left(\Omega\right)\right)\right|.
\]
\end{definition}
Area-strict convergence is necessary for Theorem~\ref{unbddcontinuitythm} to hold: note that if the conclusion to Theorem~\ref{unbddcontinuitythm} holds with integrands $f_1(x,y,A):=|y|^p$ and $f_\varphi (x,y,A):=\varphi(x)\cdot y$ (where $\varphi\in \C_0(\Omega;\mbR^m)$ is arbitrary), then convergence of the associated functionals $\mathcal{F}_1[u_j]\to\mathcal{F}_1[u]$ and $\mathcal{F}_\varphi[u_j]\to\mathcal{F}_\varphi[u]$ implies that $\norm{u_j}_p\to\norm{u}_p$ and that $u_j\rightharpoonup u$ in $\Lp^p(\Omega;\mbR^m)$ for $p\in[1,1^*]$ ($p\in[1,1^*)$ for $d=1$). It follows from uniform convexity of $\Lp^p(\Omega;\mbR^m)$ for $p>1$ that $u_j\to u$ in $\Lp^p(\Omega;\mbR^m)$ (see~\cite{BreHaiFSP}, Proposition 3.32) and hence in $\Lp^1(\Omega;\mbR^m)$. Now all that remains is to let $f_2=\sqrt{1+|A|^2}$ and apply the conclusion of Theorem~\ref{unbddcontinuitythm} to the associated functional $\mathcal{F}_2$. 

It is an immediate consequence of Theorem~\ref{unbddcontinuitythm} (with the function $f=|A|$) that area-strict convergence implies strict convergence. The following example shows that the converse is not true.
\begin{example}\label{areastrictbeatsstrict} The sequence $u_j(x)=x+(2\pi j)^{-1}\sin(2\pi j x)$ defined on $(0,1)$, converges strictly to $u=x$ but not area-strictly. To see this, note that, since $u'_j\geq 0$,
\[
 |Du_j|(0,1)=\int_0^1 u'_j(x)\;\dd x=\int_0^1 1+\sin(2\pi j x)\;\dd x \to 1,
\]
whilst, because the function $x\mapsto\sqrt{x}$ is concave on $\mbR^+$,
\[
 \int_0^1\sqrt{1+|u'_j(x)|^2}\;\dd x\geq\frac{1}{2}\int_0^1\sqrt{2}\;\dd x+\frac{1}{2}\int_0^1\sqrt{2|u'_j(x)|^2}\;\dd x=\frac{1}{\sqrt{2}}+\frac{1}{\sqrt{2}}\int u'_j(x)\;\dd x\to \sqrt{2}.
\]
\end{example}
\begin{proposition}
 Under the topology induced by area-strict convergence, $\C^\infty(\Omega;\mbR^m)$ (and hence also $\W^{1,1}(\Omega;\mbR^m)$) is dense in $\BV(\Omega;\mbR^m)$.
\end{proposition}
A proof can be found in~\cite{Bild03CVP}. We note that area-strict convergence can be interpreted as requiring strict convergence of the graph $(x,u(x))$ of $u$. Although area-strict convergence is necessary for Theorem~\ref{unbddcontinuitythm} to hold, it is only used in the proof of Lemma~\ref{continuitythm}. For every other argument in this paper, strict convergence suffices.

\section{Liftings and a continuous embedding}\label{secliftings}

In this section, we will first define a map $\mu\colon \BV(\Omega;\mbR^m)\to\mbfM(\Omega\times\mbR^m;\mbR^{m\times d})$ assigning a lifting $\mu[u]$ to each $u\in \BV(\Omega;\mbR^m)$. Our interest in liftings stems from the fact that, for positively $1$-homogeneous integrands, they can be used to compute $\mathcal{F}$ and hence, after an application of Reshetnyak's Continuity Theorem, reduce the question of the strict continuity of $\mathcal{F}$ to that of the strict continuity of the map $u\mapsto\mu[u]$. In this context, liftings were first defined and studied in~\cite{JunJer04SCML}, where the authors also note that strict continuity of the map $u\mapsto\mu[u]$ implies strict convergence of $\mathcal{F}$ for positively $1$-homogeneous integrands. We will define liftings in a different (although equivalent) way and, as a consequence, provide a cleaner derivation of the properties of liftings that we require.

Second, we will prove an embedding result for $\BV(\Omega;\mbR^m)$ equipped with the strict topology which will be needed to prove Theorem~\ref{unbddcontinuitythm} for the critical case $p=1^*$.
\begin{definition}[Liftings]\label{lifting}
For $u\in \BV\left(\Omega;\mbR^m\right)$, define for $\mathcal{H}^{d-1}$-a.e. $x\in\Omega$ the measure $\nu_x\in\mbfM(\mbR^m;\mbR^{m\times d})$ via the Riesz Representation Theorem as the functional which acts on elements $\varphi\in \C_0(\mbR^m)$ as follows:
\begin{equation}\label{nuacts}
 \int_{\mbR^m}\varphi(y)\;\dd\nu_x(y)=\frac{\dd Du}{\dd|Du|}(x)\int_0^1\varphi(u^\theta\left(x\right))\;\dd\theta.
\end{equation}
Since the jump averaging function $u^\theta$ (see Definition~\ref{jumpaverage}) and the polar function $\frac{\dd D u}{\dd|Du|}$ are defined $\mathcal{H}^{d-1}$-almost everywhere and $|Du|$-almost everywhere respectively, we have that $\nu:=(\nu_x)$ is a $|Du|$-measurable parametrised measure. The \textbf{lifting} $\mu[u]\in \mbfM\left(\Omega\times\mbR^m;\mbR^{m\times d}\right)$ is then defined to be the generalised product
\begin{equation*}
 \mu[u]:=|Du|\otimes\nu.
\end{equation*}
\end{definition}

Clearly, $\frac{\dd\mu[u]}{\dd|\mu[u]|}=\frac{\dd Du}{\dd|Du|}$, so, for positively $1$-homogeneous $f$, it is easy to calculate
\begin{align*}
 \int_{\Omega\times\mbR^m}f\left(x,y,\frac{\dd\mu[u]}{\dd |\mu[u]|}(x,y)\right)\dd|\mu[u]|(x,y) & 
=\int_{\Omega}\int_{\mbR^m}f\left(x,y,\frac{\dd Du}{\dd|Du|}(x)\right)\dd|\nu_x|(y)\;\dd|Du|(x) \\ &
=\int_{\Omega}\int_0^1 f\left(x,u^\theta\left(x\right),\frac{\dd Du}{\dd|Du|}(x)\right)\dd\theta\;\dd|Du|(x).
\end{align*}
This expression simplifies further to
\begin{align*}
&\int_{\Omega}\int_0^1 f\left(x,u^\theta\left(x\right),\frac{\dd Du}{\dd|Du|}(x)\right)\dd\theta\;\dd|Du|(x)\\ &
\qquad=\int_\Omega f\left(x,u(x),\nabla u(x)\right)\;\dd x+\int_{\Omega}\int_0^1 f\left(x,u^\theta\left(x\right),\frac{\dd D^su}{\dd|D^su|}(x)\right)\dd\theta\;\dd|D^su|(x)=\mathcal{F}[u],
\end{align*}
which explains our interest in the measure $\mu[u]$. 

The following continuity lemma is crucial to our work. It was originally established in~\cite{JunJer04SCML} using results from~\cite{BrCoLi86HMD}, but we provide a streamlined, more direct, and shorter proof here.
\begin{lemma}\label{mugoodies}
 If $u_j\to u$ strictly in $\BV(\Omega;\mbR^m)$, then $\mu[u_j]\to\mu[u]$ strictly in $\mbfM(\Omega\times\mbR^m;\mbR^{m\times})$.
\end{lemma}
\begin{proof}
We have that $|\mu[u_j]|(\Omega\times\mbR^m)=\pi_\sharp|\mu[u_j]|(\Omega)=|Du_j|(\Omega)$ and so the sequence $(\mu[u_j])$ is bounded in $\mbfM\left(\Omega\times\mbR^m;\mbR^{m\times d}\right)$. By the sequential Banach--Alaoglu Theorem, $(\mu[u_j])$ admits a weakly* convergent subsequence, which we do not relabel. Denote the limit of this sequence by $\eta$. We will show that $\eta=\mu[u]$ and, since the argument will apply to any weakly* convergent subsequence of $(\mu[u_j])$, it must follow that $\mu[u_j]\to\mu[u]$ strictly as required.

 For a fixed $\varphi\in \C^1_0(\Omega\times\mbR^m)$, we define, for $u\in \BV\left(\Omega;\mbR^m\right),\lambda\in\mbfM(\Omega\times\mbR^m;\mbR^{m\times d})$,
\begin{equation}
Q_\varphi\left(u,\lambda\right):= \int_\Omega\nabla_x\varphi\left(x,u(x)\right)\;\dd x+\int_{\Omega\times\mbR^m}\nabla_y\varphi\left(x,y\right)\;\dd\lambda(x,y).
\end{equation}
Considering $Q_\varphi\left(u_j,\mu[u_j]\right)$, we note that this expression can be rewritten using Vol'pert's averaged superposition and the chain rule in $\BV(\Omega;\mbR^m)$ (Definition~\ref{volpert} and Theorem~\ref{chainrule} in Section~\ref{prelims}) :
\begin{align*}
Q_\varphi\left(u_j,\mu[u_j]\right) & =\int_\Omega\nabla_x\varphi\left(x,u_j(x)\right)\;\dd x+\int_{\Omega}\int_0^1\nabla_y\varphi\left(x,u_j^\theta(x)\right) \; \dd \theta \; \dd Du_j(x) \\ &
=\int_\Omega\overline{\varphi}_{w_j}(x) \; \dd Dw_j(x) \\ &
=\int_\Omega D\left(\varphi\circ w_j\right),
\end{align*}
where $w_j(x)=\left(x,u_j(x)\right)$. Since $\varphi\circ w_j$ is compactly supported, we can (by mollification) approximate it with a sequence of compactly supported smooth functions whose derivatives converge strictly to $D(\varphi\circ w_j)$. The Divergence Theorem holds for each term in this sequence and so, making use of Reshetnyak's Continuity Theorem to deduce that the integrals converge, we see that 
\begin{equation}\label{equ_jmu_jchainrule}
Q_\varphi\left(u_j,\mu[u_j]\right)=0
\end{equation}
and that, analogously,
\[
Q_\varphi\left(u,\mu[u]\right)=0.
\]
Taking the limit as $j\to\infty$ in \eqref{equ_jmu_jchainrule} we deduce that $Q_\varphi\left(u,\eta\right)=0$ by weak* convergence of $\mu[u_j]$ to $\eta$ and the fact that $u_j\to u$ in $\Lp^1$ implies that $\nabla_x\varphi\left(\frarg,u_j(\frarg)\right)\to\nabla_x\varphi\left(\frarg,u(\frarg)\right)$ in $\Lp^1$. Hence, we have that $Q_\varphi(u,\mu[u])=0=Q_\varphi(u,\eta)$, which therefore implies
\begin{equation}\label{eqetamuidentity}
 \int_{\Omega\times\mbR^m}\nabla_y\varphi(x,y)\;\dd\mu[u](x,y)= \int_{\Omega\times\mbR^m}\nabla_y\varphi(x,y)\;\dd\eta(x,y).
\end{equation}
Now, let $f\in \C_c^\infty\left(\Omega\right)$ be arbitrary, $h_R\in \C_c^\infty\left(\mbR^m\right)$ be a smooth function satisfying $h_R(y)=1$ for $|y|\leq R$, $h_R(y)=0$ for $|y|\geq 2R$, and let $z\in\mathbb{B}^m$ also be arbitrary. Setting $\varphi(x,y)=f(x)h_R(y)\ip{y}{z}$ in Equation~\eqref{eqetamuidentity}, letting $R\to\infty$ and applying the Dominated Convergence Theorem, we see that
\begin{equation*}
 \ip{\int_{\Omega\times\mbR^m}f(x)\;\dd\mu[u](x,y)}{z}= \ip{\int_{\Omega\times\mbR^m}f(x)\;\dd\eta(x,y)}{z}.
\end{equation*}
By the arbitrariness of $z$, this then clearly implies
\begin{equation*}
 \int_{\Omega\times\mbR^m}f(x)\;\dd\mu[u](x,y)= \int_{\Omega\times\mbR^m}f(x)\;\dd\eta(x,y),
\end{equation*}
and hence that $\pi_\sharp \eta=Du$. Next, for $A\in\mathcal{B}(\Omega)$ (recall that $\mathcal{B}(\Omega)$ denotes the Borel sets in $\Omega$),
\begin{align*}
\pi_\sharp|\eta|\left(A\right)=|\eta|\left(A\times\mbR^m\right) & =\sup\left\{\sum_{h=0}^\infty |\eta\left(B_h\right)|:\left(B_h\right)\subset\mathcal{B}\left(\Omega\times\mbR^m\right)\text{ is a partition of }A\times\mbR^m.\right\} \\ &
\geq\sup\left\{\sum_{h=0}^\infty|\eta\left(A_h\times\mbR^m\right)|:\left(A_h\right)\subset\mathcal{B}\left(\Omega\right)\text{ is a partition of }A.\right\} \\ &
=\sup\left\{\sum_{h=0}^\infty|Du(A_h)|:\left(A_h\right)\subset\mathcal{B}\left(\Omega\right)\text{ is a partition of }A.\right\}=|Du|\left(A\right),
\end{align*}
where we used the fact that $\pi_\sharp\eta=Du$ to move from the second line to the third. This implies that $\pi_\sharp|\eta|\geq|Du|$. By the strict convergence of $(Du_j)$ and the lower semicontinuity of the total variation, we also have $\pi_\sharp|\eta|\left(\Omega\right)=|\eta|\left(\Omega\times\mbR^m\right)\leq\liminf_j|\mu[u_j]|(\Omega\times\mbR^m)=\liminf_j|Du_j|(\Omega)=|Du|\left(\Omega\right)$. Together, these two inequalities imply that $\pi_\sharp|\eta|=|Du|$.

By the Disintegration of Measures Theorem, we can therefore write $\eta=\pi_\sharp|\eta|\otimes\rho=|Du|\otimes\rho$, where $\rho\colon\Omega\to\mbfM(\mbR^m;\mbR^{m\times d})$ is a measure-valued map such that $|Du|$-almost every $|\rho_x|$ is a probability measure. Let $\varphi=f\cdot g$ where $f\in \C_0^1(\Omega)$ and $g\in \C_0^1(\mbR^m)$ are arbitrary. Varying $f$ through a countable dense subset of $\C_0^1(\Omega)$ in Equation~\eqref{eqetamuidentity}, we deduce that, for every $g\in \C_0^1(\mbR^m)$,
\begin{equation}\label{eqidentifydivs}
\int_{\mbR^m}\nabla g(y)\;\dd\rho_x(y)=\int_{\mbR^m}\nabla g(y)\;\dd\nu_x(y)=\int_0^1\nabla g(u^\theta(x))\;\dd\theta\frac{\dd Du}{\dd|Du|}(x)\quad\text{ for }|Du|\text{-a.e. }x\in\Omega,
\end{equation}
where $\nu_x$ is defined $|Du|$-a.e as in Definition~\ref{lifting}. It remains to show that $\rho_x=\eta_x$ $|Du|$-a.e. Since $|Du|$-almost every $x\in\Omega$ is either a point of approximate continuity for $u$ or a jump point of $u$, we consider these two cases separately.\\
\textbullet\, Case 1. $x\in\Omega\setminus\mathcal{J}_u$:

In this situation we can assume that $u$ is approximately continuous at $x$ and that (since it is defined $|Du|$-almost everywhere) the polar function $\frac{\dd Du}{\dd|Du|}(x)$ exists. Hence, $u^\theta(x)=u(x)$, and Equation~\eqref{eqidentifydivs} simplifies to the statement
\[
\int_{\mbR^m}\nabla g(y)\;\dd \rho_x(y)=\nabla g(u(x))\frac{\dd Du}{\dd|Du|}(x).
\]
Let $g\in C_0^1(\mbR^m)$ be such that $\norm{\nabla g}_\infty\leq 1$ and $|\nabla g(u(x))\frac{\dd Du}{\dd|Du|}(x)|=1$. Defining $g_\lambda(y):=\frac{1}{\lambda}g(u(x)+\lambda(y-u(x)))$ and noting that $\nabla g_\lambda(u(x))=\nabla g(u(x))$, $\norm{\nabla g_\lambda}_\infty=\norm{\nabla g}_\infty$, we can use Equation~\eqref{eqidentifydivs} to obtain
\[
\int_{\supp g_\lambda}\nabla g_\lambda(y)\;\dd\rho_x(y)=\int_{\mbR^m}\nabla g(y)\;\dd\rho_x(y)
\]
for every $\lambda>0$. Taking $\lambda\to\infty$ (so that $\mathbbm{1}_{\supp g_\lambda}\nabla g_\lambda\to\mathbbm{1}_{\{u(x)\}}\nabla g(u(x))$ pointwise) and using the Dominated Convergence Theorem, we can deduce
\[
\nabla g(u(x)) \rho_x\left(\{u(x)\}\right)=\nabla g(u(x))\frac{\dd Du}{\dd|Du|}(x).
\]
By our choice of $g$, this implies that $|\rho_x(\{u(x)\})|=1$ and hence, since $|\rho_x|$ is a probability measure, that $|\rho_x|=\delta_{u(x)}$. Equation~\eqref{eqidentifydivs} then easily implies that $\frac{\dd\rho_x}{\dd|\rho_x|}(u(x))=\frac{\dd Du}{\dd|Du|}(x)$, which concludes the proof in this case.\\
\textbullet\,Case 2. $x\in\mathcal{J}_u$: 

We can assume that $\frac{\dd Du}{\dd|Du|}=\frac{u^+(x)-u^-(x)}{|u^+(x)-u^-(x)|}\otimes\tau_u(x)$, where $\tau_u(x)\in\mbR^d$ is a normal vector to $\mathcal{J}_u$ at $x$. Using Equation~\eqref{eqchainruleidentity}, Equation~\eqref{eqidentifydivs} can then be written as
\begin{align*}
\int_{\mbR^m}\nabla g(y)\;\dd\rho_x(y)&=\int_0^1\nabla g(u^\theta(x))\;\dd\theta\,\frac{\dd Du}{\dd|Du|}(x)\\&=
\int_0^1\nabla g(u^\theta(x))\;\dd\theta\,\frac{(u^+(x)-u^-(x))}{|u^+(x)-u^-(x)|}\otimes\tau_u(x)\\&=
\frac{1}{|u^+(x)-u^-(x)|}\left(\int_0^1\nabla g(u^\theta(x))(u^+(x)-u^-(x))\;\dd\theta\right)\tau_u(x)
\\&=\frac{1}{|u^+(x)-u^-(x)|}\left(\int_0^1\frac{\dd}{\dd\theta}g(u^\theta(x))\;\dd\theta\right)\tau_u(x)=\frac{g(u^+(x))-g(u^-(x))}{|u^+(x)-u^-(x)|}\tau_u(x).
\end{align*}
Lemmas~\ref{lemdivmassonline} (applied to $\rho_x$) and~\ref{lemdivonlineidentify} (applied to $\rho_x-\nu_x$) below will demonstrate that, combined with the fact that $|\rho_x|\in\mbfM^1(\mbR^m)$, the above identity suffices to show that $\rho_x=\nu_x$ in $\mbfM(\mbR^m;\mbR^{m\times d})$ as required and hence that $\eta=\mu[u]$. 
\end{proof}

\begin{lemma}\label{lemdivmassonline}
Let $a,b\in\mbR^m$ with $a\neq b$, $c\in\mathbb{S}^{d-1}$ and let $\rho\in\mbfM(\mbR^m;\mbR^{m\times d})$ be such that for every $\varphi\in \C_0^1(\mbR^m)$,
\[
\int_{\mbR^m}\nabla g(y)\;\dd\rho(y)=\frac{g(b)-g(a)}{|b-a|}c,
\]
(stated equivalently, $-\nabla\cdot\rho=\frac{\delta_b-\delta_a}{|b-a|}c$ in $\mbfM(\mbR^m;\mbR^d)$). Then it must hold that
\[
|\rho|([a,b])=|\rho|(\mbR^m),
\]
where $[a,b]$ denotes the (closed) straight line segment between $a$ and $b$.
\end{lemma}

\begin{proof}
First, we define the vector-valued measure $\ip{\rho}{c}\in\mbfM(\mbR^m;\mbR^m)$ by 
\[
\int_{\mbR^m}h(y)\cdot\dd\ip{\rho}{c}(y):=\int_{\mbR^m}h(y)\otimes c: \dd\rho(y)\qquad\text{for } h\in\C(\mbR^m;\mbR^m),
\]
where $A:B:=\sum_{i,j}A^i_jB^i_j$ denotes the Frobenius product of two $m\times n$ matrices. If $C\subset\mbR^m$ is a Borel set, then, for any $g\in \C_0^1(\mbR^m)$ satisfying $\norm{\nabla g}_\infty\leq 1$ and $\nabla g\restrict C=0$, we have that
\begin{align*}
|\rho|(\mbR^m\setminus C)\geq\left|\ip{\rho}{c}\right| (\mbR^m\setminus C)&\geq\int_{\mbR^m}(\nabla g(y))^\top\cdot\dd\ip{\rho}{c}(y)\\
&=\int_{\mbR^m}\left((\nabla g(y))^\top\otimes c\right):\dd\rho(y)\\
&=\int_{\mbR^m}\nabla g(y)\;\dd\rho(y)\cdot c\\
&=\frac{g(b)-(a)}{|b-a|}c\cdot c=\frac{g(b)-(a)}{|b-a|}.
\end{align*}
Hence, if such a $g$ can be found which also satisfies $g(b)-g(a)=|b-a|$, it must follow that
\[
|\rho|(C)=|\rho|(\mbR^m)-|\rho|(\mbR^m\setminus C)\leq 1-1=0,
\]
i.e. $|\rho|(C)=0$. We will use an approximate version of this strategy to show that, for every $y_0\in\mbR^m\setminus[a,b]$, there exists a $\delta>0$ such that $|\rho|(B_\delta(y_0))=0$. By considering the union of $B_{\delta_i}(y_i)$ across a countable dense set $\{y_i\}\subset\mbR^m\setminus[a,b]$, we will then obtain that $|\rho|(\mbR^m\setminus[a,b])=0$.

Fix $y_0\in\mbR^m\setminus[a,b]$, let $p$ be the closest point to $y_0$ on $[a,b]$ and define
\[
g_{y_0}(y):=\begin{cases}
|y-a|-|p-a| &\text{ if }|y-a|\leq|p-a|,\\
|b-p|-|y-b| & \text{ if }|y-b|\leq|b-p|,\\
0 & \text{ otherwise.}
\end{cases}
\]
The function $g_{y_0}$ is continuous and compactly supported, is piecewise differentiable on $\mbR^m$ with $\norm{\nabla g_{y_0}}_\infty\leq 1$ and (since $p\in[a,b]$) satisfies $g_{y_0}(b)-g_{y_0}(a)=|b-p|+|p-a|=|b-a|$. By our choice of $p$, $|y_0-p|\leq|y_0-a|$ and $|y_0-p|\leq|y_0-b|$, whence it follows that $g_{y_0}(y_0)=0$. If $p\in(a,b)$, then (since the closest point to $y_0$ on $[a,b]$ must be unique) these inequalities must be strict and we can deduce that there exists a $\delta>0$ such that $g_{y_0}\restrict B_\delta(y_0)\equiv 0$. If $p\in\{a,b\}$, a similar line of reasoning applies and the same conclusion follows.

Now let $\kappa$ be a smooth, positive mollifier and consider the mollifications $g_{y_0,\varepsilon}:=g_{y_0}*\kappa_\varepsilon$. These functions are smooth, have support contained within $\supp g_{y_0}+B_\varepsilon(0)$, and, since mollifications of continuous functions converge pointwise, it holds that $g_{y_0,\varepsilon}(b)-g_{y_0,\varepsilon}(a)\to|b-a|$ as $\varepsilon\to 0$. In addition, since $\nabla g_{y_0,\varepsilon}=(\nabla g_{y_0})*\kappa_\varepsilon$, it also holds that $\norm{\nabla g_{y_0,\varepsilon}}_\infty\leq\norm{\nabla g_{y_0}}_\infty\int_{\mbR^m}\kappa_\varepsilon(y)\;\dd y=1$. For $\varepsilon<\delta/2$ we have that $g_{y_0,\varepsilon}\restrict B_{\delta/2}(y_0)\equiv 0$, and hence that
\[
|\rho|(\mbR^m\setminus B_{\delta/2}(y_0))\geq\lim_{\varepsilon\to 0}\int_{\mbR^m}\nabla g_{y_0,\varepsilon}(y)\;\dd\rho(y)\cdot c=\lim_{\varepsilon\to 0}\frac{g_{y_0,\varepsilon}(b)-g_{y_0,\varepsilon}(a)}{|b-a|}=1,
\]
as required.
\end{proof}

The following lemma is a special case of Theorem~D.1 from~\cite{BrCoLi86HMD}.
\begin{lemma}\label{lemdivonlineidentify}
Let $a,b\in\mbR^m$ be such that $a\neq b$ and let $\rho\in\mbfM(\mbR^m;\mbR^{m\times d})$ be such that $|\rho|(\mbR^m\setminus[a,b])=0$ and that, for every $g\in \C_0^1(\mbR^m)$,
\[
\int_{\mbR^m}\nabla g(y)\;\dd\rho(y)=0.
\]
Then it must hold that
\[
\rho=0\quad\text{ in }\mbfM(\mbR^d;\mbR^{m\times d}).
\]
\end{lemma}

\begin{proof}
Let $g\in\C^1_0(\mbR^m)$ be arbitrary, $z\in[a,b]$ be such that $|\rho|(\{z\})=0$ and, for a given $\varepsilon>0$, let $\eta\in\C^1_0(\mbR^m)$ be such that $0\leq\eta\leq 1$, $\eta\equiv 1$ on $[a,z]$, $\eta\equiv 0$ on $[z+\varepsilon(b-a),b]$ and $\norm{\nabla\eta}_\infty\leq 2/(\varepsilon|b-a|)$. We then have that
\begin{align*}
0&=\int_{\mbR^m}\nabla(\eta g)(y)\;\dd\rho(y)=\int_a^b\eta(y)\nabla g(y)\;\dd\rho(y)+\int_a^b g(y)\nabla\eta(y)\;\dd\rho(y)\\&=
\int_a^z\nabla g(y)\;\dd\rho(y)+\int_z^{z+\varepsilon(b-a)}\eta(y)\nabla g(y)\;\dd\rho(y)+\int_{z}^{z+\varepsilon(b-a)}g(y)\nabla\eta(y)\;\dd\rho(y).
\end{align*}
We claim that the final two integrals tend to $0$ as $\varepsilon\to 0$. For the middle integral, this is immediate since the assumption that $z$ is not an atom of $\rho$ implies that
\[
\left|\int_z^{z+\varepsilon(b-a)}\eta(y)\nabla g(y)\;\dd\rho(y)\right|\leq\norm{\nabla g}_\infty|\rho|([z,z+\varepsilon(b-a)])\to 0
\]
as $\varepsilon\to 0$. For the second integral, we can use the fact that $g$ is Lipschitz to observe
\begin{align*}
&\left|\int_{z}^{z+\varepsilon(b-a)}g(y)\nabla\eta(y)\;\dd\rho(y)\right|\\&
\qquad\leq\int_z^{z+\varepsilon(b-a)}\left|g(y)-g(z)\right|\cdot\norm{\nabla\eta}_\infty\;\dd|\rho|(y)
+\left|g(z)\int_z^{z+\varepsilon(b-a)}\nabla\eta(y)\;\dd\rho(y)\right|\\&
\qquad\leq\norm{\nabla g}_\infty|b-a|\varepsilon\left(\frac{2}{\varepsilon|b-a|}\right)|\rho|([z,z+\varepsilon(b-a)])
+\norm{g}_\infty\left|\int_{\mbR^m}\nabla\eta(y)\;\dd\rho(y)\right|\\&
\qquad=2\norm{\nabla g}_\infty|\rho|\left([z,z+\varepsilon(b-a)]\right)\to 0.
\end{align*}
Since $\rho$ can have at most countably many atoms, we have therefore obtained that $\int_a^z\nabla g(y)\;\dd \rho(y)=0$ for all but countably many $z\in[a,b]$. Taking $g$ to be affine in a neighbourhood of $[a,b]$, we see that this implies $\int_{z_1}^{z_2}w\;\dd\rho(y)=0$ for all but countably many $z_1,z_2\in[a,b]$ for every $w\in\pd\mathbb{B}^m$, from which it follows that $\rho=0$ as required.
\end{proof}

\begin{remark}
The authors of~\cite{JunJer04SCML} define the \textbf{minimal lifting} of $u$ to be the measure $\mu[u]$ satisfying the equation $Q_\varphi(u,\mu[u])=0$ with the additional property that $\pi_\sharp|\mu[u]|(\Omega)=|Du|(\Omega)$. This is equivalent to our definition of a lifting.
\end{remark}

The following corollary is now a direct consequence of Reshetnyak's Continuity Theorem and Lemma~\ref{mugoodies}.

\begin{corollary}\label{1homogenouscontinuity}
Let $f\in \C(\Omega\times\mbR^m\times\mbR^{m\times d})$ be positively $1$-homogeneous in the final variable and satisfy $|f(x,y,A)|\leq C|A|$. Then the functional 
\begin{align*}
\mathcal{F}[u]&=\int_{\Omega\times\mbR^m}f\left(x,y,\frac{\dd\mu[u]}{\dd |\mu[u]|}(x,y)\right)\dd|\mu[u]|(x,y)\\
&=\int_\Omega f(x,u(x),\nabla u(x))\;\dd x+ \int_\Omega\int_0^1 f^\infty\left(x,u^\theta (x),\frac{\dd D^s u}{\dd|D^s u|}(x)\right)\;\dd\theta\;\dd|D^su|(x)
\end{align*}
is strictly continuous on $\BV(\Omega;\mbR^m)$.
\end{corollary}
\begin{proof}
Simply combine Corollary~\ref{mugoodies} with Reshetnyak's Continuity Theorem~\ref{Reshetnyak}, the discussion following Definition~\ref{lifting}, and the fact that $|f(x,y,A)|\leq C|A|$ implies that the restriction of $f$ to $\Omega\times\mbR^m\times\pd\mathbb{B}^{m\times d}$ is bounded.
\end{proof}
Next, we prove an embedding result for the space $\BV(\Omega;\mbR^m)$ equipped with the metric of strict convergence, which will be of use in Section~\ref{bigproof}.

\begin{proposition}\label{ctsembedding}
Let $\Omega\subset\mbR^d$ be a domain with compact Lipschitz boundary and assume that $d>1$. Then the embedding
\[\BV(\Omega;\mbR^m)\hookrightarrow \Lp^{1^*}(\Omega;\mbR^m)
\]
is continuous when $\BV(\Omega;\mbR^m)$ is equipped with the topology of strict convergence.
\end{proposition}
This result is of interest since it yields an extension of the continuous embedding $\BV(\Omega;\mbR^m)\hookrightarrow \Lp^p(\Omega;\mbR^m)$ for all $p<1^*$ when $\BV(\Omega;\mbR^m)$ is equipped with the usual weak* topology to the critical case $p=1^*$. Note that Proposition~\ref{ctsembedding} does not hold when $d=1$, as Example~\ref{excriticalcasefails} in Section~\ref{bigproof} demonstrates.
\begin{proof}
Let $u_j\to u$ strictly. Since ($u_j)$ converges to $u$ in measure (as a consequence of strong $\Lp^1$ convergence), this then implies, via Vitali's Convergence Theorem, that $u_j\to u$ in $\Lp^{1^*}$ if and only if $(u_j)$ is $1^*$-uniformly integrable. In this situation, assuming that $(u_j)$ is not $1^*$-uniformly integrable, we can apply Lions' concentration-compactness principle (Lemma I.1 in~\cite{Lion85CCP}) to arrive at a contradiction. For reasons of clarity, however, we will carry out the derivation here in full:

If $(u_j)_{j\in\mathbb{N}}$ is not $1^*$-uniformly integrable, $(u_j-u)_{j\in\mathbb{N}}$ is not either. Extending $u\in \BV(\Omega;\mbR^m)$ by zero to an element of $\BV(\mbR^d;\mbR^m)$, we have
\[
 |Du|=|Du|\restrict\Omega+|u_{\pd\Omega}|\mathcal{H}^{d-1}\restrict\pd\Omega,
\]
where $u_{\pd\Omega}\in \Lp^1(\pd\Omega;\mbR^m)$ is the trace of $u$ on $\pd\Omega$. Since the map $u\mapsto u_{\pd\Omega}\in \Lp^1(\pd\Omega;\mbR^m)$ is strictly continuous (see~\cite{AmFuPa00FBVF}, Theorem 3.88), we have that $(u_j)$ is strictly convergent in $\BV(\Omega;\mbR^m)$ if and only if it is in $\BV(\mbR^d;\mbR^m)$. Without loss of generality, then, we can view $(u_j)$ and $u$ as elements of $\BV(\mbR^d;\mbR^m)$ whose support is contained in the compact set $\overline{\Omega}$. By the Rellich--Kondrachov Theorem, $(u_j)$ is bounded in $\Lp^{1^*}(\Omega;\mbR^m)$ and so, by passing to a subsequence, we can also assume that 
\begin{equation*}
|u_j-u|^{1^*}\wsc\gamma \text{ and } |D(u_j-u)|\wsc\nu \text{ in } \mbfM\left(\overline{\Omega}\right).
\end{equation*}
 Since $(|u_j-u|^{1^*})$ is not uniformly integrable and is supported in $\overline{\Omega}$, we can (via Prokhorov's Theorem) assume that $\gamma>0$.

Recall that the Poincar\'e inequality
\begin{equation}\label{gagliardonirenberg}
 \norm{u}_{1^*}\leq C|Du|\left(\mbR^d\right)
\end{equation}
holds for $u$ in $\BV(\mbR^d;\mbR^m)$ (see Section~5.6.1 in~\cite{EvaGar92MTFP}). For $\varphi\in \C^1(\overline{\Omega})$, we can apply the inequality \eqref{gagliardonirenberg} to $\varphi u_j$ in order to obtain:
\[
 \norm{\varphi u_j}^{1^*}_{1^*}\leq C\left(|D(\varphi u_j)|(\mbR^d)\right)^{1^*}\leq  C\left(\int_{\overline{\Omega}} |\varphi(x)|\;\dd|Du_j|(x)+\int_{\overline{\Omega}} |\nabla \varphi(x)||u_j(x)|\;\dd x\right)^{1^*}.
\]
By passing again to a subsequence, we can assume that $(u_j)$ converges pointwise $\mathcal{L}^d$-almost everywhere. Taking the limit as $j\to\infty$ whilst using the Brezis--Lieb Lemma on the left hand side and strict convergence of $u_j$ on the right hand side (note that Reshetnyak's Continuity Theorem implies that if $Du_j\to Du$ strictly, then $|Du_j|\wsc|Du|$ as well) we obtain
\begin{equation}\label{gamma|Du|inequality}
\begin{aligned}
 \norm{\varphi u}^{1^*}_{1^*} +\int_{\overline{\Omega}}|\varphi(x)|^{1^*}\;\dd\gamma(x) & =\norm{\varphi u}^{1^*}_{1^*}+\lim_{j\to\infty}\norm{\varphi u_j-\varphi u}_{1^*}^{1^*}\\
& =\lim_{j\to\infty}\norm{\varphi u_j}_{1^*}^{1^*}\\
& \leq C\left(\int_{\overline{\Omega}}|\varphi(x)|\;\dd|Du|(x)+\int_{\overline{\Omega}}|\nabla\varphi(x)||u(x)|\;\dd x\right)^{1^*}.
\end{aligned}
\end{equation}
We will show that $\gamma$ consists only of atoms and that $\gamma\ll|Du|$, which leads to a contradiction.

Applying \eqref{gagliardonirenberg} to $\varphi(u-u_j)$, we see
\[
 \norm{\varphi(u-u_j)}_{1^*}^{1^*}\leq C\left(\int_{\overline{\Omega}}|\varphi (x)|\left|D(u-u_j)\right|(x)+\int_{\overline{\Omega}}|\nabla\varphi(x)||(u-u_j)(x)|\;\dd x\right)^{1^*}.
\]
Letting $j\to\infty$ and using $u_j\to u$ in $\Lp^1(\Omega;\mbR^m)$, then using $\varphi$ to approximate the indicator function of a generic Borel set $A$ gives
\begin{equation}\label{gammanuinequality}
 \gamma(A)\leq C\left(\nu(A)\right)^{1^*}.
\end{equation}
It follows that $\gamma\ll\nu$. For an arbitrary $x\in\overline{\Omega}$, Equation \eqref{gammanuinequality} implies the key nonlinear estimate
\[
 \frac{\gamma\left(B_r(x)\right)}{\nu\left(B_r(x)\right)}\leq C\left(\nu(B_r(x)\right)^{\frac{1}{d-1}}\qquad\text{whenever }\nu\left(B_r(x)\right)>0
\]
and so, taking $r\downarrow 0$ and using the Besicovitch Derivation Theorem, we have that
\[
 \frac{\dd\gamma}{\dd\nu}(x)=0
\]
unless $x$ is an atom of $\nu$. Since $\gamma$ is absolutely continuous with respect to $\nu$ and is finite, we can therefore deduce that $\gamma=\sum_{i\in I}\gamma_i\delta_{x_i}$ for some countable set $I$, some summable sequence $(\gamma_i)\subset\mbR^+$ and some sequence $(x_i)\subset\overline{\Omega}$ of distinct points. Since $\gamma>0$, at least one of the $\gamma_i$ must be nonzero, say $\gamma_0$ at the point $x_0$. Let $\varphi\in \C^\infty_c(B(0,1))$ be such that $0\leq \varphi\leq 1$ and $\varphi(0)=1$. Using $\varphi_\varepsilon:=\varphi(\frac{x-x_0}{\varepsilon})$ in \eqref{gamma|Du|inequality}, we deduce
 
\begin{align*}
&\int_{B(x_0,\varepsilon)}\left(|\varphi_\varepsilon(x)|| u(x)|\right)^{1^*}\;\dd x +\int_{B(x_0,\varepsilon)}|\varphi_\varepsilon(x)|^{1^*}\;\dd\gamma(x) \\ 
&\qquad \leq C\left(\int_{B(x_0,\varepsilon)}|\varphi_\varepsilon(x)|\;\dd|Du|(x) +\int_{B(x_0,\varepsilon)}|\nabla\varphi_\varepsilon(x)||u(x)|\;\dd x\right)^{1^*} 
\\
&\qquad \leq C\left(\int_{B(x_0,\varepsilon)}|\varphi_\varepsilon(x)|\;\dd|Du|(x) + \left(\int_{B(x_0,\varepsilon)}|\nabla\varphi_\varepsilon(x)|^d\;\dd x\right)^{{1}/{d}}\left(\int_{B(x_0,\varepsilon)}|u(x)|^{1^*}\;\dd x\right)^{{1}/{1^*}}\right)^{1^*}
\\
& \qquad= C\left(\int_{B(x_0,\varepsilon)}|\varphi_\varepsilon(x)|\;\dd|Du|(x) +\left(\int_{B(0,1)}|\nabla\varphi(y)|^d\;\dd y\right)^{1/d}\left(\int_{B(x_0,\varepsilon)}|u(x)|^{1^*}\;\dd x\right)^{1/1^*}\right)^{1^*}
\\
&\qquad=C\left(\int_{B(x_0,\varepsilon)}|\varphi_\varepsilon(x)|\;\dd|Du|(x) + \norm{\nabla\varphi}_{\Lp^d}\norm{u}_{\Lp^{1^*}(B(x_0,\varepsilon))}\right)^{1^*}.
\end{align*}
Letting $\varepsilon\downarrow 0$, we obtain
\[
 0<\gamma_0\leq \left(|Du|({x_0})\right)^{1^*},
\]
which is impossible for $d>1$, since the derivatives of $\BV(\Omega;\mbR^m)$ functions must vanish on singletons. Hence, $|u_j|$ must be $1^*$-uniformly integrable and so the result is proved.
\end{proof}

\begin{remark}
 The fact that $u_j\to u$ in $\Lp^{1^*}(\Omega;\mbR^m)$ whenever $u_j\to u$ area-strictly in $\BV(\Omega;\mbR^m)$ is a necessary consequence of Theorem~\ref{unbddcontinuitythm} in the case $p=1^*$. Letting $f(x,y,A)=|y|^{1^*}$, we see that Theorem~\ref{unbddcontinuitythm} implies $\norm{u_j}_{1^*}\to\norm{u}_{1^*}$ whenever $u_j\to u$ area-strictly. Since $\Lp^{1^*}$ is a uniformly convex space and $u_j\rightharpoonup u$ in $\Lp^{1^*}$ (a consequence of the fact that $(u_j)$ is bounded in $\Lp^{1^*}$ and that $u_j$ converges to $u$ in measure), we therefore have that $u_j\to u$ in $\Lp^{1^*}$ (see, for example, Proposition 3.32 of~\cite{BreHaiFSP}).
\end{remark}

\section{Perspective functions and area-strict convergence}\label{secperints}
The purpose of this section is to remove the $1$-homogeneity assumption which appears in Corollary~\ref{1homogenouscontinuity}. This is achieved by introducing a \emph{perspective function} $\tilde{f}$ for the integrand $f$ and exchanging strict convergence for area-strict convergence. We note here that a similar approach applying Reshetnyak's theorems combined with perspective functions to integral functionals on $\BV(\Omega;\mbR)$ can be found in~\cite{DMas79IRBV}. For a discussion of generalised perspective functions and their relevance to different notions of convexity, the reader is referred to~\cite{DacMar08PFIC}.

\begin{definition}[The perspective function]\label{defperfun}
Let $f\in\C(\Omega\times\mbR^m\times\mbR^{m\times d})$ be a map whose recession function $f^\infty$ exists. The \textbf{perspective function} $\tilde{f}\colon\Omega\times\mbR^{1+m}\times\mbR^{(1+m)\times d}\to\mbR$ of $f$ is defined by
\begin{equation}
 \tilde{f}\left(x,(r,y),(t,A)\right):=
\begin{cases}
|t|f\left(x,y,|t|^{-1}A\right) & \text{ if }|t|\neq 0\\
f^\infty\left(x,y,A\right) & \text{ if }|t|=0,
\end{cases}
\end{equation}
for $x\in\Omega$, $(r,y)\in\mbR\times\mbR^m\cong\mbR^{1+m}$ and $(t,A)\in\mbR^d\times\mbR^{m\times d}\cong\mbR^{(1+m)\times d}$.
\end{definition}
Strictly speaking, the perspective function of $f$ is only unique as an element of $\C(\Omega\times(\mbR\times\mbR^m)\times(\mbR^d\times\mbR^{m\times d}))$ where, for realisation as an element of $\C(\Omega\times\mbR^{1+m}\times\mbR^{(1+m)\times d})$ the canonical identifications $\mbR\times\mbR^m\cong\mbR^{1+m}$, $\mbR^d\times\mbR^{m\times d}\cong\mbR^{(1+m)\times d}$. We will tacitly assume that such a choice has been made and will speak simply of `the' perspective function.

It follows immediately from Definition \ref{defperfun} that $\tilde{f}$ is always positively $1$-homogeneous in the $(t,A)$ argument. The following lemma shows that $\tilde{f}$ inherits the continuity properties of $f$.
\begin{lemma}\label{lemlinperfuncts}
 Let $f\in \C(\Omega\times\mbR^m\times\mbR^{m\times d})$ be such that $f^\infty$ exists. Then $\tilde{f}\in \C(\Omega\times\mbR^{1+m}\times\mbR^{(1+m)\times d})$.
\end{lemma}

\begin{proof}
 That $\tilde{f}$ is continuous away from where $|t|=0$ is an immediate consequence of the continuity of $f$. Continuity of $\tilde{f}$ when $|t|=0$ follows directly from the definition of the recession function.
\end{proof}

The following construction, which essentially replaces $u(x)$ with its graph $(x,u(x))$, combined with Lemma~\ref{lemlinperfuncts}  allows us to remove the $1$-homogeneity assumption from Corollary~\ref{1homogenouscontinuity}:

For $u\in \BV(\Omega;\mbR^m)$, we define $U\in \BV(\Omega;\mbR^{1+m})$ by
\[
 U(x):=(|x|,u(x)).
\]
The sequence $(U_j)\subset \BV(\Omega;\mbR^{1+m})$ is defined analogously. From the derivative decomposition
\begin{align*}
DU&=\left(\frac{x}{|x|}\mathcal{L}^d\restrict\Omega,Du\right)=\left(\frac{x}{|x|},\nabla u\right)\mathcal{L}^d\restrict\Omega+(0,D^su),\\
|DU|&=\sqrt{\left|\frac{x}{|x|}\right|^2+|\nabla u|^2}\mathcal{L}^d\restrict\Omega +|D^su|=\sqrt{1+|\nabla u|^2}\mathcal{L}^d\restrict\Omega +|D^su|,
\end{align*}
it follows that
\begin{align*}
 u_j\wsc u \text{ in }\BV(\Omega;\mbR^m) & \iff U_j\wsc U\text{ in }\BV(\Omega;\mbR^{1+m}),\\
u_j\to u\text{ area-strictly in }\BV(\Omega;\mbR^m) & \iff U_j\to U\text{ strictly in }\BV(\Omega;\mbR^{1+m}).
\end{align*}
Upon computing, we find that
\begin{equation}
\begin{aligned} 
& \int_\Omega\tilde{f}\left(x,U(x),\frac{\dd DU}{\dd|DU|}(x)\right)\;\dd|DU|(x) \\&\qquad= \int_\Omega \tilde{f}\left(x,(|x|,u(x)),\frac{\big(\frac{x}{|x|},\nabla u(x)\big)}{\sqrt{1+|\nabla u(x)|^2}}\right)\sqrt{1+|\nabla u(x)|^2}\;\dd x \\ &\qquad\qquad\qquad\qquad+
\int_\Omega\int_0^1\tilde{f}\left(x,(|x|,u^\theta(x)),\left(0,\frac{\dd D^s u}{\dd|D^s u|}(x)\right)\right)\;\dd|D^s u|(x) \\ = &
\int_\Omega \tilde{f}\left(x,(|x|,u(x)),\left(\frac{x}{|x|},\nabla u(x)\right)\right)\;\dd x +\int_\Omega\int_0^1\tilde{f}\left(x,(|x|,u^\theta(x)),\left(0,\frac{\dd D^s u}{\dd|D^s u|}(x)\right)\right)\;\dd|D^s u|(x)\\ = &
\int_\Omega f\left(x,u(x),\nabla u(x)\right)\;\dd x +\int_\Omega\int_0^1 f^\infty\left(x,u^\theta (x),\frac{\dd D^s u}{\dd|D^s u|}(x)\right)\;\dd|D^s u|(x).
\end{aligned}
\end{equation}
This then immediately leaves us with a proof of Theorem~\ref{unbddcontinuitythm} in the case where $f(x,\frarg,y)$ is bounded:

\begin{lemma}\label{continuitythm}
 Let $f\in \C\left(\Omega\times\mbR^m\times\mbR^{m\times d}\right)$ satisfy $|f(x,y,A)|\leq C(1+|A|)$ and be such that $f^\infty$ exists. Then, the functional
\begin{equation*}
 \mathcal{F}[u]:=\int_\Omega f\left(x,u(x),\nabla u(x)\right)\;\dd x +\int_\Omega\int_0^1 f^\infty\left(x,u^\theta(x),\frac{\dd D^su}{\dd|D^su|}(x)\right)\;\dd\theta\;\dd|D^su|(x)
\end{equation*}
is $\langle\frarg\rangle$-strictly continuous on $\BV\left(\Omega;\mbR^m\right)$.
\end{lemma}
\begin{proof}
 Simply use Lemma~\ref{lemlinperfuncts} to apply Corollary~\ref{1homogenouscontinuity} to the map
\[
U\mapsto \int_\Omega\tilde{f}\left(x,U(x),\frac{\dd DU}{\dd|DU|}(x)\right)\;\dd|DU|(x).
\]
\end{proof}

\section{Approximation arguments and unbounded integrands}\label{bigproof}
This section contains the final step in the proof of Theorem~\ref{unbddcontinuitythm} and also Theorem~\ref{notcts}, an extension of Theorem~\ref{unbddcontinuitythm} to Carath\'eodory integrands, as well as a counterexample to show that the hypotheses of Theorem~\ref{notcts} are optimal.

Lemma~\ref{continuitythm} has allowed us to show that $\mathcal{F}$ is area-strictly continuous whenever $f$ is continuous, satisfies $|f(x,y,A)|\leq C(1+|A|)$ and is such that $f^\infty$ exists. \color{black} To prove Theorem~\ref{unbddcontinuitythm} in the case where $f(x,\frarg,y)$ is unbounded,\color{black} we will approximate a general $f\in \C(\Omega\times\mbR^m\times\mbR^{m\times d})$ by a sequence of integrands $f_k\to f$ which satisfy the hypotheses of Lemma~\ref{continuitythm}. This approximation leaves the following remainder, the control of which is sufficient to prove the theorem:
\[
 \int_{\left\{|u_j(x)|\geq k\right\}}1+|u_j(x)|^p+|\nabla u_j(x)|\;\dd x,\qquad p\in[1,1^*]\text{ }([1,1^*)\text{ if }d=1).
\]
To control the first term in the integrand it suffices to use the fact that $u_j\to u$ in measure. To control the second, we will use the strong $\Lp^p$-convergence of $u_j$ to $u$: as a consequence of the Rellich-Kondrachov Theorem, we automatically have that $u_j$ converges to $u$ in $\Lp^p$ for $p\in[1,1^*)$, $\Lp^{1^*}$ convergence for the critical case is obtained via Proposition~\ref{ctsembedding}. We will use liftings in order to apply Reshetnyak's Lower Semicontinuity Theorem as a means of controlling the third term.

\begin{proof}[Proof of Theorem~\ref{unbddcontinuitythm}]
 Let $f$ satisfy the given hypotheses and let $u_j\to u$ area-strictly in $\BV(\Omega;\mbR^m)$. Let $(\varphi_n)_{n=1}^{\infty}$ be a smooth partition of unity of $\mbR^m$ such that each $\varphi_n$ has compact support with
\[
\begin{array}{ll}
\sum_{n=1}^k\varphi_n(y)=1 & \text{ whenever } |y|< k,\\
\sum_{n=1}^k\varphi_n(y)=0 & \text{ whenever } |y|\geq k+1.
\end{array}
\]
We define the sequence $(f_k)$ by 
\[
f_k(x,y,A)=\left(\sum_{n=1}^k\varphi_n(y)\right)f(x,y,A).
\]
By construction, $|f_k(x,y,A)|\leq C(1+k+|A|)$, and so each $f_k$ is in $\C(\overline{\Omega}\times\mbR^m\times\mbR^{m\times d})$ and satisfies the hypotheses of Lemma~\ref{continuitythm}. 

Define the functional $\mathcal{F}_k\colon \BV(\Omega;\mbR^m)\to\mbR$ by
\[
\mathcal{F}_k[u]:=\int_\Omega f_k\left(x,u(x),\nabla u(x)\right)\;\dd x +\int_\Omega\int_0^1 f_k^\infty\left(x,u^\theta(x),\frac{\dd D^su}{\dd|D^su|}(x)\right)\;\dd\theta\;\dd|D^su|(x),
\]
so that we can estimate

\begin{equation}
\begin{aligned}
 \lim_{j\to\infty}|\mathcal{F}[u]-\mathcal{F}[u_j]|&\leq \liminf_{k\to\infty}\lim_{j\to\infty}|\mathcal{F}[u_j]-\mathcal{F}_k[u_j]| \\ &
\qquad+\lim_{k\to\infty}\lim_{j\to\infty}|\mathcal{F}_k[u_j]-\mathcal{F}_k[u]|+\lim_{k\to\infty}|\mathcal{F}_k[u]-\mathcal{F}[u]|.
\end{aligned}
\end{equation}
By the Dominated Convergence Theorem, the final term tends to $0$ as $k\to\infty$ and, by Lemma~\ref{continuitythm}, $\lim_{j\to\infty}\mathcal{F}_k[u_j]=\mathcal{F}_k[u]$ for every $k\in\mathbb{N}$ and so the second term is equal to $0$. Hence, in order to prove area-strict continuity of $\mathcal{F}$, we need only control the first term. 

Assume for simplicity that $(u_j)\subset \C^\infty(\Omega;\mbR^m)$. Since $\C^\infty(\Omega;\mbR^m)$ is dense in $\BV(\Omega;\mbR^m)$ with respect to the area-strict topology, this is not a restrictive assumption: once the result is proved for convergent sequences of smooth functions, we can use a diagonal argument to show that our argument holds for any area-strictly convergent sequence $(u_j)$. Now consider
\begin{align}\label{importantbound}
 |\mathcal{F}[u_j]-\mathcal{F}_k[u_j]| & \leq\int_\Omega|f\left(x,u_j(x),\nabla u_j(x)\right)-f_k\left(x,u_j(x),\nabla u_j(x)\right)|\;\dd x \nonumber\\ & =
\int_\Omega \left(\sum_{n=k+1}^\infty\varphi_n(u_j(x))\right)|f\left(x,u_j(x),\nabla u_j(x)\right)|\;\dd x  \nonumber\\ & \leq
\int_{\left\{|u_j(x)|\geq k\right\}}|f\left(x,u_j(x),\nabla u_j(x)\right)|\;\dd x \nonumber \\ & \leq
C\int_{\left\{|u_j(x)|\geq k\right\}}\left(1+|u_j(x)|^p+|\nabla u_j(x)|\right)\;\dd x.
\end{align}
Since $u_j\to u$ strongly in $\Lp^p$ (because of Rellich Kondrachov if $p<1^*$ and Proposition~\ref{ctsembedding} if $p=1^*$), the sequence $(u_j)$ is $p$-uniformly integrable. We also have that $u_j\to u$ in measure (as a consequence of just $\Lp^1$ convergence), and so the first two terms in the final integrand vanish uniformly in $j$ as $k\to\infty$. We are left, then, with the task of controlling the term
\[
\lim_{j\to\infty}\int_{\left\{|u_j(x)|\geq k\right\}}|\nabla u_j(x)|\;\dd x.
\]
We can rewrite this integral in terms of the lifting $\mu[u_j]$,
\[
\int_{\left\{|u_j(x)|\geq k\right\}}|\nabla u_j(x)|\;\dd x= \int_{\Omega\times\mbR^m}\mathbbm{1}_{\mbR^m\setminus{B_k(0)}}(y)\;\dd|\mu[u_j]|(x,y),
\]
where $B_k(0)$ is the open ball in $\mbR^m$ of radius $k$ centered at $0$. 
Let $\mu[u_j]=\mu_j,\mu[u]=\mu$ and note that these satisfy the hypotheses of Reshetnyak's Lower Semicontinuity Theorem. Since $\mbR^m\setminus{B_k(0)}$ is closed, the function $(y,A)\mapsto(1-\mathbbm{1}_{\mbR^m\setminus{B_k(0)}}(y))|A|$ is lower semicontinuous and positively 1-homogeneous and convex in the second variable. Applying Reshetnyak's Lower Semicontinuity Theorem, then, we see that
\begin{equation}\label{lowersemicontinuitywithmu}
 \int_{\Omega\times\mbR^m}(1-\mathbbm{1}_{\mbR^m\setminus B_k(0)}(y))\;\dd|\mu[u]|(x,y)\leq\liminf_{j\to\infty}\int_{\Omega\times\mbR^m}(1-\mathbbm{1}_{\mbR^m\setminus B_k(0)}(y))\;\dd|\mu[u_j]|(x,y).
\end{equation}
Since $\mu[u_j]\to\mu[u]$ strictly, we have that
\[
 \int_{\Omega\times\mbR^m}1\;\dd|\mu[u]|(x,y)=\lim_{j\to\infty}\int_{\Omega\times\mbR^m}1\;\dd|\mu[u_j]|(x,y).
\]
Hence, \eqref{lowersemicontinuitywithmu} becomes
\[
 \int_{\Omega\times\mbR^m}\mathbbm{1}_{\mbR^m\setminus B_k(0)}(y)\;\dd|\mu[u]|(x,y)\geq\limsup_{j\to\infty}\int_{\Omega\times\mbR^m}\mathbbm{1}_{\mbR^m\setminus B_k(0)}(y)\;\dd|\mu[u_j]|(x,y).
\]
Since $\mathbbm{1}_{\mbR^m\setminus B_k(0)}(y)\to 0$ pointwise, the Dominated Convergence Theorem implies that the left hand side of this final inequality must tend to $0$ as $k\to\infty$. The result is hence proved.
\end{proof}

The condition \eqref{assumption1} does not make sense when $d=1$, whereby $1^*=\infty$, and the ``natural'' requirement in this case is the nonlocal condition
\[
 \mathcal{F}[u]\leq C(1+\norm{u}_\infty+|Du|(\Omega)).
\]
However, the following example demonstrates that Proposition~\ref{ctsembedding} does not hold in this case and hence that allowing $p$-growth for $p<\infty$ only is optimal for $d=1$.
\begin{example}\label{excriticalcasefails}
 Let $\Omega=(-1,1)$ and define $(u_j)\in \BV(\Omega;\mbR)$ by
\begin{equation*}
 u_j(x):=\begin{cases}
          -1 & \text{ if }x\in(-1,-1/j),\\
	  jx & \text{ if }x\in [-1/j,1/j],\\
	  1 & \text{ if }x\in(1/j, 1).
         \end{cases}
\end{equation*}
We clearly have that $u_j$ converges area-strictly to the function $u:=-\mathbbm{1}_{(-1,0]}+\mathbbm{1}_{(0,1)}$, but not uniformly since (however $u$ is defined at $x=0$) $\sup_{x\in[-1/j,1/j]}|u_j(x)-u(x)|=1$. Hence, Theorem~\ref{unbddcontinuitythm} does not hold for the functional $\mathcal{F}[u_j]=\norm{u_j-u}_\infty$.
\end{example}

Finally, we will weaken our assumptions on the regularity of $f$ and $f^\infty$. The proof of Theorem~\ref{notcts} proceeds by using the Scorza Dragoni theorem to determine the result when $f$ is bounded. An approximation argument is then used to extend this result to the case where $f^\infty\equiv 0$ (ie, when $f$ has `negligible growth at $\infty$'). Applying this result to $f-f^\infty$ lets us deduce the general result.

\begin{theorem}\label{notcts}
Let $u_j\to u$ area-strictly in $\BV(\Omega;\mbR^m)$ and let $f\colon\Omega\times\mbR^m\times\mbR^{m\times d}\to\mbR$ be a Carath\'eodory integrand satisfying \eqref{assumption1} whose recession function $f^\infty$ exists on the set $(\Omega\setminus N)\times\mbR^m\times\mbR^{m\times d}$, where $N$ is some Borel set satisfying $(\mathcal{L}^d+|Du|)(N)=0$. Then it holds that $\mathcal{F}[u_j]\to\mathcal{F}[u]$.
\end{theorem}
In particular, this theorem implies that $\mathcal{F}$ is area-strictly continuous for any Carath\'eodory $f$ where $f^\infty$ exists on $(\Omega\setminus N)\times\mbR^m\times\mbR^{m\times d}$ for some Borel set $N$ with $\mathcal{H}^{d-1}(N)=0$.
\begin{proof}
As in the proof of Theorem~\ref{continuitythm}, we start with the estimate
\begin{equation*}
\begin{aligned}
 \lim_{j\to\infty}|\mathcal{F}[u]-\mathcal{F}[u_j]| & \leq \liminf_{k\to\infty}\lim_{j\to\infty}|\mathcal{F}[u_j]-\mathcal{F}_k[u_j]|\\ &\qquad+\lim_{k\to\infty}\lim_{j\to\infty}|\mathcal{F}_k[u_j]-\mathcal{F}_k[u]|+\lim_{k\to\infty}|\mathcal{F}_k[u]-\mathcal{F}[u]|,
\end{aligned}
\end{equation*}
where $\mathcal{F}_k$ is defined as before. To control the first term, we can repeat the approximation argument used in the proof of Theorem~\ref{continuitythm} exactly, since the estimate \eqref{importantbound} remains valid under our new hypotheses on $f$.

As a consequence of our assumptions on $f$ and $f^\infty$, the functions $x\mapsto f_k(x,u(x),\nabla u(x))$ and $x\mapsto\int_0^1 f^\infty_k(x,u^\theta(x),\frac{\dd Du}{\dd|Du|}(x))\,\dd\theta$ are defined $\mathcal{L}^d$ and $|D^su|$-almost everywhere respectively and converge pointwise $\mathcal{L}^d$ (or $|D^s u|$)-almost everywhere to $x\mapsto f(x,u(x),\nabla u(x))$ and $x\mapsto\int_0^1 f^\infty(x,u^\theta(x),\frac{\dd Du}{\dd|Du|}(x))\,\dd\theta$ as $k\to\infty$. Hence, as before, we can deduce as from the Dominated Convergence Theorem that
\[
 \lim_{k\to\infty}|\mathcal{F}_k[u]-\mathcal{F}[u]|=0.
\]
Consequently, we need only control the second term in the estimate above. It suffices, therefore, to consider the case where $|f(x,y,A)|\leq C(1+|A|)$. We will complete the proof for this case in three steps:

First, assume that $f$ is bounded, so that $|f(x,y,A)|\leq C$ for all $(x,y,A)\in\Omega\times\mbR^m\times\mbR^{m\times d}$. By the Scorza--Dragoni Theorem, there exists a compact set $K_\varepsilon\subset\Omega$ such that $\mathcal{L}^d(\Omega\setminus K_\varepsilon)\leq \varepsilon$ and that $f\restrict(K_\varepsilon\times\mbR^m\times\mbR^{m\times d}$) is continuous. By the Tietze Extension Theorem, we can find a continuous function $g\in \C(\Omega\times\mbR^m\times\mbR^{m\times d})$ which restricts to $f$ on $K_\varepsilon\times\mbR^m\times\mbR^{m\times d}$. Moreover, by truncating $g$ outside of $K_\varepsilon\times\mbR^m\times\mbR^{m\times d}$ if needs be, we can assume that $g$ is bounded and that $\norm{g}_\infty=\norm{f}_\infty$. We note that Theorem~\ref{unbddcontinuitythm} applies to $g$, so that
\[
 \int_\Omega g(x,u_j(x),\nabla u_j(x))\;\dd x\to\int_\Omega g(x,u(x),\nabla u(x))\;\dd x
\]
for any sequence $(u_j)$ converging area-strictly to $u$ in $\BV(\Omega;\mbR^m)$. However, we also have that
\begin{align*}
 &\bigg|\int_\Omega f(x,u_j(x),\nabla u_j(x))-g(x,u_j(x),\nabla  u_j(x))\; \dd x\bigg| \\ &\qquad\leq\int_{\Omega\setminus K_\varepsilon}|f(x,u_j(x),\nabla u_j(x))|+|g(x,u_j(x),\nabla u_j(x))|\;\dd x \leq 2\varepsilon\norm{f}_\infty,
\end{align*}
and so
\[
 \int_\Omega f(x,u_j(x),\nabla u_j(x))\;\dd x\to \int_\Omega f(x,u(x),\nabla u(x))\;\dd x
\]
 as well.

Next, assume that $f$ satisfies $|f(x,y,A)|\leq C(1+|A|)$, but that $f^\infty\equiv 0$. This implies that, for every $\varepsilon>0$, there exists $R>0$ such that $|f(x,y,A)|\leq\varepsilon(1+|A|)$ whenever $|A|\geq R$. Otherwise, we would have a sequence of points $A_k$ such that $|A_k|\to\infty$ where $|f(x,y,A_k)|\geq\varepsilon(1+|A_k|)$. The sequence $(A_k/(1+|A_k|))$ must have a convergent subsequence, converging to some limit $\tilde{A}\in\mathbb{B}^{m\times d}$. Taking the limit in $f(x,y,A_k)/(1+|A_k|)$ we would then have that $f^\infty(x,y,\tilde{A})\geq\varepsilon$, a contradiction. Now take $\varphi_\varepsilon\in \C_0^\infty (\mbR^{m\times d})$ satisfying $0\leq\varphi_\varepsilon\leq 1$, $\varphi_\varepsilon\equiv 1$ on $B(0,R)$, $\varphi_\varepsilon\equiv 0$ on $\mbR^{m\times d}\setminus B(0,R+1)$ in order to define $f_\varepsilon=\varphi_\varepsilon f$. By construction, $f_\varepsilon$ is bounded on $\Omega\times\mbR^m\times\mbR^{m\times d}$ and so, by the previous step, the associated functional $\mathcal{F}_\varepsilon$ is continuous. We also see, however, that
\begin{align*}
 \left|\int_\Omega f_\varepsilon(x,u_j(x),\nabla u_j(x))-f(x,u_j(x),\nabla u_j(x))\;\dd x\right| & \leq\int_{\{|\nabla u(x)|\geq R\}}|f(x,u_j(x),\nabla u_j(x))|\;\dd x \\ & \leq \int_{\{|\nabla u_j(x)|\geq R\}}\varepsilon(1+|\nabla u_j(x)|)\;\dd x \\ & \leq \varepsilon(|\Omega|+\norm{\nabla u_j}_1)
\end{align*}
and so the functional associated to $f$ is area-strictly convergent as well.

Now, assume that $f$ satisfies $f(x,y,A)|\leq C(1+|A|)$ and that $f^\infty$ is defined and continuous on $(\Omega\setminus N)\times\mbR^m\times\mbR^{m\times d}$ for some Borel set $N$ satisfying $(\mathcal{L}^d+|Du|)(N)=0$. Setting $h=f-f^\infty$ and applying our work from the previous step, we see that it is sufficient to show that
\begin{align*}
 &\int_{(\Omega\setminus N)\times\mbR^m}f^\infty\left(x,y,\frac{\dd\mu[u_j]}{\dd|\mu[u_j]|}(x,y)\right)\;\dd|\mu[u_j]|(x,y)\\ &\qquad\to \int_{(\Omega\setminus N)\times\mbR^m}f^\infty\left(x,y,\frac{\dd\mu[u]}{\dd|\mu[u]|}(x,y)\right)\;\dd|\mu[u]|(x,y).
\end{align*}
This is trivial, however, since the discontinuity set, $N$, of $f^\infty$ is negligible with respect to the limit measure $|\mu[u]|$ which implies that Reshetnyak's Continuity Theorem still holds in this case (see, for example, Proposition 1.62 and the proof of Theorem 2.39 in~\cite{AmFuPa00FBVF}), and so the theorem is proved.
\end{proof}
Finally, we finish with an example which shows that Theorem~\ref{notcts} is optimal, in the sense that $f^\infty$ cannot be discontinuous on a set that is charged by $|Du|$ if we are to expect area-strict continuity from $\mathcal{F}$.
\begin{example}
 Let $\Omega=(-1,1)$ and define the sequence $(u_j)\in \BV(\Omega;\mbR)$ by $u_j:=\mathbbm{1}_{[-1/j,1)}$, so that $(u_j)$ converges area-strictly to $u=\mathbbm{1}_{[0,1)}$. Define the function $g\colon\Omega\times\mbR\to\mbR$ by
\[
 g(x,A):=|A|\begin{cases}
          0 &\text{if }\quad x\leq-1/|A|,\\
	  |A|(x+1/|A|)&\text{if }\quad-1/|A|<x<0,\\
	  1 &\text{if }\quad x\geq 0.
         \end{cases}
\]
Now let $h\in \C_c^\infty(\mbR)$ be such that $h\equiv 1$ on $[-1,1]$ and define
\[
 f(x,A):=(1-h(A))g(x,A).
\]
We have that $f$ is a continuous function and that $f^\infty(x,A)=\mathbbm{1}_{[0,1)}(x)|A|$. We can therefore compute
\[
 \mathcal{F}[u_j]=0+\int_{-1}^{1}\mathbbm{1}_{[0,1)}\;\dd\delta_{-1/j}=0,
\]
whereas
\[
 \mathcal{F}[u]=\mathbbm{1}_{[0,1)}\;\dd\delta_{0}=1.
\]
\end{example}


\begin{thebibliography}{10}

\bibitem{AmCiFu07RRBV}
M.~Amar, V.~{De Cicco}, and N.~Fusco.
\newblock {\em A relaxation result in {BV} for integral functionals with
  discontinuous integrands.}
\newblock {ESAIM Control Optim. Calc. Var.}, 13:396--412, 2007.

\bibitem{AmbDal92RBVQ}
L.~Ambrosio and G.~{Dal Maso}.
\newblock {\em On the relaxation in {${\rm BV}(\Omega;{\bf R}\sp m)$} of}
  quasi-convex integrals.
\newblock {J. Funct. Anal.}, 109:76--97, 1992.

\bibitem{AmFuPa00FBVF}
L.~Ambrosio, N.~Fusco, and D.~Pallara.
\newblock {\em {Functions of Bounded Variation and Free-Discontinuity
  Problems}}.
\newblock Oxford Mathematical Monographs. Oxford University Press, 2000.

\bibitem{AmbPal93IRRF}
L.~Ambrosio and D.~Pallara.
\newblock {\em Integral representations of relaxed functionals on $\BV(\mbR^n,\mbR^k)$ and polyhedral approximation.}
\newblock Indiana Univ. Math. J. 42 (1993), no. 2, 295–321

\bibitem{AviGig91VIMB}
P.~Aviles and Y.~Giga.
\newblock {\em Variational integrals on mappings of bounded variation and their lower semicontinuity.}
\newblock {Arch. Rational Mech. Anal.}, 115:201--255, 1991.

\bibitem{Bild03CVP}
M.~Bildhauer.
\newblock {\em Convex variational problems}, volume 1818 of {\em Lecture Notes
  in Mathematics}.
\newblock Springer, 2003.

\bibitem{BreHaiFSP}
H.~Brezis.
\newblock {\em Functional analysis, {S}obolev spaces and partial differential
  equations}.
\newblock Universitext. Springer, 2011.

\bibitem{BrCoLi86HMD}
H.~Br{\'e}zis, J.-M.~Coron and E.~Lieb.
\newblock {\em Harmonic maps with defects.}
\newblock {Proc. Amer. Math. Soc.}, 107:649--705, 1986.

\bibitem{BreLie83RBC}
H.~Br{\'e}zis and E.~Lieb.
\newblock {\em A relation between pointwise convergence of functions and convergence
  of functionals.}
\newblock {Comm. Math. Phys.}, 88:486--490, 1983.

\bibitem{Daco08DMCV}
B.~Dacorogna.
\newblock {\em {Direct Methods in the Calculus of Variations}}, volume~78 of
  {\em Applied Mathematical Sciences}.
\newblock Springer, 2nd edition, 2008.

\bibitem{DacMar08PFIC}
B.~Dacorogna and P.~Mar{\'e}chal.
\newblock {\em The role of perspective functions in convexity, polyconvexity, rank-one convexity and separate convexity.}
\newblock {J. Convex Anal.}, 15:271--284, 2008.

\bibitem{DMas79IRBV}
G.~Dal~Maso.
\newblock {\em Integral representation on {${\rm BV}(\Omega )$} of {$\Gamma$}-limits
  of variational integrals.}
\newblock {Manuscripta Math.}, 30:387--416, 1980.

\bibitem{CiFuVe05LSB}
V.~De Cicco, N.~Fusco and A.~Verde.
\newblock {\em On {$\Lp^1$}-lower semicontinuity in {BV}.}
\newblock {J. Convex Anal.}, 12:173--185, 2005.

\bibitem{Dell91LSCF}
S.~Delladio.
\newblock {\em Lower semicontinuity and continuity of functions of measures with
  respect to the strict convergence.}
\newblock {Proc. Roy. Soc. Edinburgh Sect. A}, 119:265--278, 1991.

\bibitem{Evan90WCMN}
L.~C. Evans.
\newblock {\em Weak convergence methods for nonlinear partial differential
  equations}, volume~74 of {\em CBMS Regional Conference Series in
  Mathematics}.
\newblock Conference Board of the Mathematical Sciences, 1990.

\bibitem{EvaGar92MTFP}
L.~C Evans and R.~F Gariepy.
\newblock {\em Measure theory and fine properties of functions}
\newblock Studies in Advanced Mathematics. CRC Press, Boca Raton, FL, 1992.

\bibitem{FonMul93RQFB}
I.~Fonseca and S.~M{\"u}ller.
\newblock {\em Relaxation of quasiconvex functionals in {${\rm BV}(\Omega,{\bf R}\sp
  p)$} for integrands {$f(x,u,\nabla u)$}.}
\newblock {Arch. Ration. Mech. Anal.}, 123:1--49, 1993.

\bibitem{JunJer04SCML}
R.~L. Jerrard and N.~Jung.
\newblock {\em Strict convergence and minimal liftings in {$\BV$}.}
\newblock {Proc. Roy. Soc. Edinburgh Sect. A}, 134:1163--1176, 2004.

\bibitem{KriRin10RSI}
J.~Kristensen and F.~Rindler.
\newblock {\em Relaxation of signed integral functionals in {BV}.}
\newblock {Calc. Var. Partial Differential Equations}, 37:29--62, 2010.

\bibitem{Lion85CCP}
P.-L. Lions.
\newblock {\em The concentration-compactness principle in the calculus of
  variations. {T}he limit case. {I}.}
\newblock {Rev. Mat. Iberoamericana}, 1:145--201, 1985.

\bibitem{Rind12LSYM}
F.~Rindler.
\newblock {\em Lower semicontinuity and {Young} measures in {BV} without {Alberti's
  Rank-One Theorem}.}
\newblock {Adv. Calc. Var.}, 5:127--159, 2012.

\bibitem{SERR59NDI}
J.~Serrin.
\newblock {\em A new definition of the integral for nonparametric problems in the
  calculus of variations.}
\newblock {Acta Math.}, 102:23--32, 1959.

\bibitem{SERR61PVI}
J.~Serrin.
\newblock {\em On the definition and properties of certain variational integrals.}
\newblock {Trans. Amer. Math. Soc.}, 101:139--167, 1961.


\bibitem{Spec11SPRR}
D.~Spector.
\newblock {\em Simple proofs of some results of {R}eshetnyak.}
\newblock {Proc. Amer. Math. Soc.}, 139:1681--1690, 2011.

\end{thebibliography}

\end{document}